\documentclass{siamart}

\usepackage{lipsum}
\usepackage{amsfonts}
\usepackage{graphicx}
\usepackage{epstopdf}
\ifpdf
  \DeclareGraphicsExtensions{.eps,.pdf,.png,.jpg}
\else
  \DeclareGraphicsExtensions{.eps}
\fi

\usepackage{algorithm}
\usepackage[noend]{algpseudocode}

\usepackage[normalem]{ulem}
\usepackage{enumitem}
\usepackage{placeins}

\usepackage{caption}
\usepackage{subcaption}

\usepackage{pgf}
\usepackage{tikz}
\usepackage{tikzscale}
\usepackage[utf8]{inputenc}
\usepackage{pgfplots}
\usepgfplotslibrary{groupplots}
\pgfplotsset{compat=newest}
\usetikzlibrary{plotmarks}
\usepackage{grffile}

\usepackage{color}
\usepackage{amsopn}
\usepackage{bm}
\usepackage{bbm}
\usepackage{amsmath, amssymb, mathtools}
\usepackage{extarrows}

\usepackage{subcaption}

\newcommand{\TheAuthors}{N. Alger, V. Rao, A. Myers, T. Bui-Thanh, O. Ghattas}

\newcommand{\ShortTitle}{Adaptive product-convolution approximation}
\newcommand{\LongTitle}{Scalable matrix-free adaptive product-convolution approximation for locally translation-invariant operators}

\newsiamthm{thm}{Theorem}
\newsiamthm{cor}{Corollary}
\newsiamthm{prop}{Proposition}
\newsiamthm{lem}{Lemma}
\newsiamthm{example}{Example}
\newsiamthm{defn}{Definition}
\newsiamthm{assume}{Assumption}
\newsiamthm{condition}{Condition}
\newsiamthm{remark}{Remark}

\newcommand{\nor}[1]{\left\| #1 \right\|}
\newcommand{\ix}[1]{\left[ #1 \right]}
\newcommand*\inds[1]{\mathbf{\bm{#1}}}

\newcommand{\nadj}{q}

\DeclareMathOperator{\corners}{corners}
\DeclareMathOperator{\supp}{supp}
\DeclareMathOperator{\nbrs}{nbrs}
\DeclareMathOperator{\leaves}{leaves}
\DeclareMathOperator{\children}{children}
\DeclareMathOperator{\cells}{cells}

\DeclareMathOperator{\flip}{flip}

\DeclareMathOperator{\cond}{cond}

\graphicspath{{./}}

\headers{\ShortTitle}{\TheAuthors}

\title{{\LongTitle}\thanks{This work was partially supported by the following grants: AFOSR
FA9550-17-1-0190, NSF ACI-1550593 and CBET-1508713, and DOE
DE-SC0010518, DE-SC0009286, and DE-SC0019393. Aaron Myers and Tan Bui-Thanh are partially supported by the DOE grant DE-SC0018147. Vishwas Rao is partially supported by the U.S. Department of Energy, Office of Science, Advanced Scientific Computing Research Program
under contract DE-AC02-06CH11357.}}

\author{
  Nick Alger\thanks{Institute for Computational Engineering and Sciences, The University of Texas at Austin, Austin, TX, USA
    (\email{nalger@ices.utexas.edu}, \email{aaron@ices.utexas.edu}, \email{tanbui@ices.utexas.edu}, \email{omar@ices.utexas.edu}).}
  \and
  Vishwas Rao\thanks{Mathematics and Computer Science Division, Argonne National Laboratory, Lemont, IL, USA (\email{vhebbur@anl.gov}).}
  \and
  Aaron Myers\footnotemark[2]
  \and
  Tan Bui-Thanh\footnotemark[2] \thanks{Department of Aerospace Engineering and Engineering Mechanics, and Institute for Computational Engineering and Sciences, The University of Texas at Austin, Austin, TX, USA.}
  \and
  Omar Ghattas\footnotemark[2] \thanks{Departments of Geological Sciences and Mechanical Engineering, The University of Texas at Austin, Austin, TX, USA.}
}

\ifpdf

\hypersetup{
  pdftitle={\LongTitle},
  pdfauthor={\TheAuthors}
}

\fi

\begin{document}

\maketitle

\begin{abstract}
We present an adaptive grid matrix-free operator approximation scheme based on a ``product-convolution'' interpolation of convolution operators. This scheme is appropriate for operators that are locally translation-invariant, even if these operators are high-rank or full-rank. Such operators arise in Schur complement methods for solving partial differential equations (PDEs), as Hessians in PDE-constrained optimization and inverse problems, as integral operators, as covariance operators, and as Dirichlet-to-Neumann maps. Constructing the approximation requires computing the impulse responses of the operator to point sources centered on nodes in an adaptively refined grid of sample points. A randomized a-posteriori error estimator drives the adaptivity. Once constructed, the approximation can be efficiently applied to vectors using the fast Fourier transform. The approximation can be efficiently converted to hierarchical matrix ($H$-matrix) format, then inverted or factorized using scalable $H$-matrix arithmetic. The quality of the approximation degrades gracefully as fewer sample points are used, allowing cheap lower quality approximations to be used as preconditioners. This yields an automated method to construct preconditioners for locally translation-invariant Schur complements. We directly address issues related to boundaries and prove that our scheme eliminates boundary artifacts. We test the scheme on a spatially varying blurring kernel, on the non-local component of an interface Schur complement for the Poisson operator, and on the data misfit Hessian for an advection dominated advection-diffusion inverse problem. Numerical results show that the scheme outperforms existing methods.
\end{abstract}



\section{Introduction}
\label{sec:intro0}

We present an adaptive product-convolution scheme for approximating locally translation-invariant operators. That is, operators $A:l^2(\inds{\Omega})\rightarrow l^2(\inds{\Omega})$ satisfying
\begin{equation}
\label{eq:translation_invariance}
A\ix{y,x} \approx A\ix{y-x+p, p}
\end{equation}
whenever $x$ is not too far from $p$ (see Figure \ref{fig:translation_invariance}). Here we consider the case in which $\inds{\Omega}$ is a box\footnote{One can use our scheme in more general settings by mapping the domain to a box and interpolating functions onto a regular mesh.} in $\mathbb{Z}^d$. Our scheme is well-suited for approximating or preconditioning operators that arise in Schur complement techniques \cite{le1997non,saad1999distributed} for solving partial differential equations (PDEs), reduced Hessians in PDE-constrained optimization and inverse problems, integral operators, covariance operators with spatially varying kernels, and Dirichlet-to-Neumann maps or other Poincaré–Steklov operators in multiphysics problems. These operators are typically dense and implicitly defined, and often do not admit a global low-rank approximation, making them difficult to approximate with standard techniques.

\begin{figure}
\center
\includegraphics[scale=0.4]{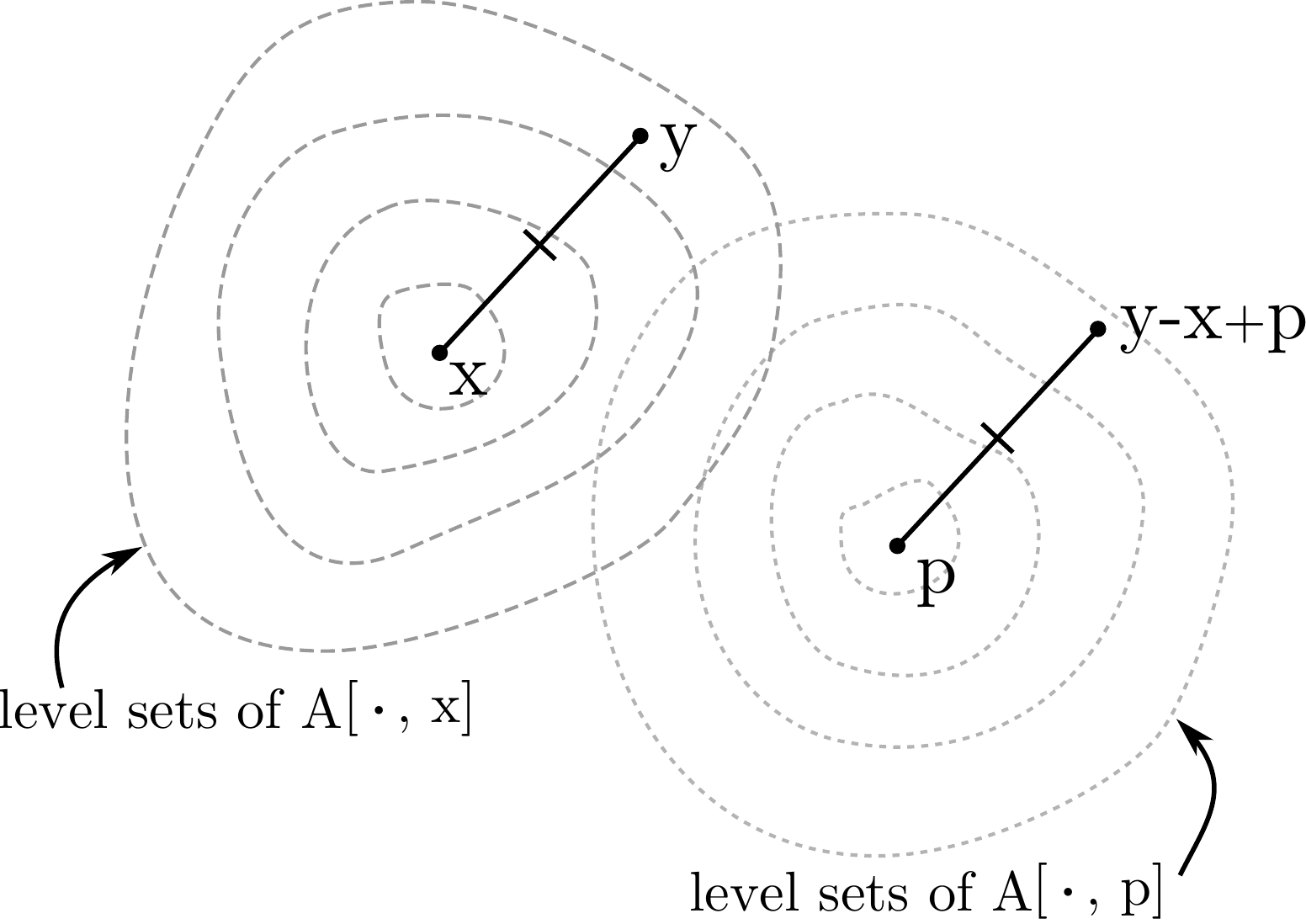}
\caption{Our product-convolution scheme is suitable for operators that are locally translation-invariant. That is, operators for which the impulse response at $y$ to a point source at $x$ is similar to the impulse response at $y-x+p$ given a point source of equal magnitude at $p$ if $x$ is near $p$.}
\label{fig:translation_invariance}
\end{figure}

Let $\varphi_p$ be the impulse response of $A$ at $p$, i.e., the function created by applying $A$ to a point source centered at point $p$, then translating the result to recenter it at $0$:
\begin{equation}
\label{eq:varphi_p}
\varphi_p\ix{z} = \left(A \delta_{p}\right)\ix{z+p}, \quad z \in \inds{\Omega}-p.
\end{equation}
By ``point source,'' $\delta_p$, we mean the Kronecker delta that contains the value $1$ at location $p$ and zeros elsewhere. If $A$ were translation-invariant (i.e., if \eqref{eq:translation_invariance} held with equality for all $x$, $y$), then $A$ would be the convolution operator $A:f \mapsto \varphi_p \ast f$. To approximate operators that are only locally translation-invariant, we patch together a collection of convolution operators, each of which well-approximates $A$ locally. Our approximation of $A$, denoted $\widetilde{A}$, takes the following form:
\begin{equation}
\label{eq:conv_op_approx}
Af \approx \widetilde{A} f \coloneqq \sum_{k=1}^r \varphi_k^E \ast (w_k \cdot f),
\end{equation}
where the $w_k$ are locally supported weighting functions that overlap and form a partition of unity, `$\cdot$' denotes pointwise multiplication of functions, $\ast$ denotes convolution (see Section \ref{sec:setting_notation} for more details on notation), and the functions $\varphi_k^E$ are modified\footnote{To address issues with boundary artifacts, we construct $\varphi_k^E$ by extending the function $\varphi_{p_k}$ outside of $\inds{\Omega}-p_k$ using information from neighboring functions, $\varphi_{p_j}$ (more on this in Section \ref{sec:boundary_extension}).} versions of the (translated, recentered) impulse responses $\varphi_{p_k}$ associated with a collection of sample points, $p_k$. Each point $p_k$ is contained within the support of the associated weighting function $w_k$. 

The basic form of \eqref{eq:conv_op_approx} is known as a product-convolution approximation, and is well-established in the literature (see Section \ref{sec:convolution_interpolation_review}). Here we improve upon existing schemes by:
\begin{itemize} 
	\item Adaptively and automatically choosing the sample points $p_k$.
	\item Addressing issues related to boundaries. 
\end{itemize}
In Section \ref{sec:convolution_operator_approx} we derive our scheme, explain how we choose $p_k$, and detail the process for constructing $w_k$ and $\varphi_k^E$. In Section \ref{sec:H_matrix} we detail how $\widetilde{A}$ can be used once constructed, including how to efficiently convert it to hierarchical matrix ($H$-matrix) format. In Section \ref{sec:error_analysis} we perform an a-priori error analysis of our scheme. We demonstrate our scheme numerically in Section \ref{sec:numerical_examples} and give concluding remarks in Section \ref{sec:conclusion}. In the remainder of this section we summarize our results (Section \ref{sec:overview}), review existing work (Section \ref{sec:literature_review}), and define our setting and notation (Section \ref{sec:setting_notation})

\subsection{Overview of results}
\label{sec:overview}
The scheme we present is matrix-free in the sense that constructing $\widetilde{A}$ only requires the ability to apply $A$ and its adjoint, $A^*$, to vectors. Access to the matrix representation of $A$ is not needed. Once constructed, we can compute any matrix entry of $\widetilde{A}$ in $O(1)$ work. We can apply $\widetilde{A}$ and $\widetilde{A}^*$ to vectors in nearly linear work using the fast Fourier transform (FFT). Blocks of $\widetilde{A}$ and $\widetilde{A}^*$ can be applied to vectors in work that is nearly linear in the size of the block. 

Often the ultimate goal is to solve linear systems with $A$ as the coefficient operator. Krylov methods can be used to solve these systems \cite{CalvettiBryanLothar00}. However, the convergence of Krylov methods depends heavily on the spectral structure of the coefficient operator, leading to slow convergence when $A$ is ill-conditioned. To address this, we explain how $\widetilde{A}$ can be efficiently converted to $H$-matrix format. Once in $H$-matrix format, $\widetilde{A}$ can be efficiently factorized or inverted using $H$-matrix arithmetic, then used as a preconditioner. Alternatively, one can build circulant preconditioners from $\widetilde{A}$ \cite{ChanNg96,NagyOleary98}.

We choose the sample points, $p_k$, in an adaptive grid: in regions where the error is large, we refine the grid. The effect of this refinement process is to place more sample points in regions where $A$ is less translation-invariant, and fewer sample points in regions where $A$ is more translation-invariant. The adaptivity is performed using a randomized a-posteriori error estimator.

Boundaries introduce two difficulties for product-convolution schemes:
\begin{enumerate}
\item \textbf{Boundary artifacts:} The impulse response associated with $p_k$ is naturally defined on $\inds{\Omega}-p_k$, but the product-convolution scheme \eqref{eq:conv_op_approx} requires it to be defined on a larger set. The three standard extension techniques---extending the impulse response by zero, reflecting it across the boundary, or replicating it periodically---all create boundary artifacts wherever artificial data are used in place of undefined data.
\item \textbf{Boundary effects:} The underlying operator may fail to be translation-invariant near boundaries due to boundary conditions or other physically meaningful effects.
\end{enumerate}
To overcome 1, we extend the support of the impulse responses using information from neighboring impulse responses. To overcome 2, we use anisotropic adaptivity. Our adaptive refinement scheme senses the coordinate direction in which $A$ is least translation-invariant within a cell, and preferentially subdivides the cell in that direction. This allows the scheme to efficiently approximate operators that are not translation-invariant in directions perpendicular to boundaries, but are translation-invariant in directions parallel to boundaries. Boundary effects due to boundary conditions typically exhibit this direction-dependent form of translation-invariance (regardless of the type of boundary condition). 

In Theorem \ref{thm:a_priori_convolution_error}, we prove that the error in our scheme is controlled by the local failure of translation-invariance in $A$. This, together with adaptivity, implies convergence: our scheme will continue to add new sample points until it achieves the desired error tolerance. The more translation-invariant $A$ is, the fewer sample points will be used. Additionally, Theorem \ref{thm:a_priori_convolution_error} implies that our approximation scheme will not introduce boundary artifacts. Without our impulse response extension procedure, the bound in Theorem \ref{thm:a_priori_convolution_error} would fail near the boundary.

We demonstrate the scheme on a spatially varying blur operator, on the non-local component of an interface Schur complement for the Poisson operator, and on the data misfit Hessian for an advection dominated advection-diffusion inverse problem. Our scheme outperforms existing methods:
\begin{itemize}
\item Our scheme converges much faster than non-adaptive product-convolution approximation for the spatially varying blur operator.
\item The number of sample points required to approximate the non-local component of the Poisson Schur complement is independent of the mesh size.
\item Approximation using a small number of sample points yields  a high quality preconditioner for the Poisson Schur complement.
\item The number of sample points required to approximate the advection-diffusion Hessian is independent of the Peclet number, a proxy for the informativeness of the data in the inverse problem. 
\item A Hessian preconditioner that results from using our approximation performs well even if the Peclet number is large.
\end{itemize} 
We also find that the randomized a-posteriori error estimator performs much better than standard theory predicts: we see that it performs almost as well with $5$ random samples as it does with $100$.

Although our scheme will eventually converge to any desired error tolerance, it is most useful for computing moderately accurate approximations (say, $80$\% to $99$\% accurate) of ``difficult'' operators that are poorly approximated by standard techniques. In our numerical tests, we observe that the convergence slows beyond this accuracy. Moderate accuracy approximation is sufficient for many engineering applications, and is ideal for building preconditioners.

\subsection{Existing work}
\label{sec:literature_review}

The most widely used, robust, and general purpose matrix-free operator approximation schemes are based on low-rank approximation (Section \ref{sec:low_rank}). However, many important operators in PDEs, PDE-constrained optimization and inverse problems, and integral equations are not low-rank. Our scheme fits within a class of operator approximation schemes based on interpolation of convolution operators (Section \ref{sec:convolution_interpolation_review}).  Hierarchical matrices (Section \ref{sec:H_matrix_background}) are another well-established operator approximation format; they are simultaneously a tool we use (Section \ref{sec:H_matrix_actual}), and an alternative to our scheme.

\subsubsection{Low-rank approximation}
\label{sec:low_rank}

Low-rank approximations---matrix factorizations of the form $A \approx B C$, where $B$ is $N \times r$ (tall), and $C$ is $r \times N$ (wide)---can be efficiently constructed in a matrix-free setting by using Krylov methods (Lanczos or Arnoldi),  randomized SVD \cite{HalkoMartinssonTropp11} or CUR decomposition/skeletonization \cite{Cheng05,goreinov1997theory,Mahoney11,Tyrtyshnikov00}. Although low-rank approximations have been used for Dirichlet-to-Neumann maps \cite{Calvetti15b,Calvetti15}, full-rank or high-rank operators typically still retain a high rank after being restricted to a boundary as a Schur complement. Likewise, although low-rank approximations have been used to approximate the (prior preconditioned) Hessian of the data misfit term in PDE-constrained inverse problems \cite{BuiEtAl13,CuiEtAl14,FlathEtAl11,PetraEtAl14,SpantiniEtAl15}, the numerical rank of this term grows as the informativeness of the data in the inverse problem grows \cite{AlgerEtAl17}, making low-rank approximation inefficient for highly informative data. Even when the operator is low-rank in the sense that $r \ll N$, the cost of computing the low-rank approximation may be prohibitive. For example, a low-rank approximation of the Hessian in a PDE constrained optimization or inverse  problem requires $O(r)$ linearized forward/adjoint PDE solves, so that for large-scale problems with e.g.\ $N$ of order $10^6$, even a compression of $0.1\%$ still means that thousands of forward solves are needed, which is often an expensive proposition \cite{Bui-ThanhBursteddeGhattasEtAl12,ChenVillaGhattas18,IsaacPetraStadlerEtAl15}. Our scheme is motivated by a desire to go beyond low-rank approximation in these applications.

\subsubsection{Convolution interpolation}
\label{sec:convolution_interpolation_review}

Since the linear operator that performs a convolution may be numerically full-rank (e.g., convolution with $\delta_0$: the identity operator) or high-rank (e.g., convolution with a Gaussian with a small width), interpolation of convolution operators can, where applicable, be used to approximate dense operators with far fewer terms than the rank of the operator.

Operator approximation schemes based on weighted sums of convolution operators with spatially varying weights (``convolution interpolations'') fall into two categories: product-convolution schemes where one performs element-wise products with weighting functions first and convolutions second, and convolution-product schemes where this order is reversed:
\begin{equation}
\label{eq:conv_prod_vs_prod_conv}
Af \approx \quad\quad \underbrace{\sum_{k=1}^r \psi_k \ast (\omega_k \cdot f)}_{\text{product-convolution}} \quad\quad \text{vs.} \quad\quad \underbrace{\sum_{k=1}^r \omega_k \cdot \left(\psi_k \ast f\right)}_{\text{convolution-product}}.
\end{equation}
The terms ``product-convolution'' and ``convolution-product'' refer to the general format of the approximations in \eqref{eq:conv_prod_vs_prod_conv}, where $\psi_k$ and $\omega_k$ could be any functions. For us, $\psi_k$ are (modified) impulse response functions and $\omega_k$ form a partition of unity. Since the entries of a convolution operator $M:f\mapsto\psi \ast f$ are $M\ix{y,x} = \psi\ix{y-x}$, product-convolution and convolution-product approximations have the following $(y,x)$ matrix entries:
\begin{equation}
\label{eq:conv_prod_vs_prod_conv_entries}
A\ix{y,x}\approx \quad\quad \underbrace{\sum_{k=1}^r \omega_k\ix{x}\psi_k\ix{y-x}}_{\text{product-convolution}} \quad\quad \text{vs.} \quad\quad \underbrace{\sum_{k=1}^r \omega_k\ix{y}\psi_k\ix{y-x}}_{\text{convolution-product}}.
\end{equation}
Both schemes are non-symmetric, but the adjoint of a product-convolution operator is a convolution-product operator, and vice versa. The operators defined by the following actions are adjoints of each other:
\begin{equation}
\label{eq:adjoint_random_z}
\underbrace{\sum_{k=1}^r \psi_k \ast \left(\omega_k \cdot f\right)}_{\widetilde{A}f} \quad \xlongleftrightarrow{\text{adjoint}} \quad \underbrace{\sum_{k=1}^r \overline{\omega}_k\cdot \left(\flip\left(\overline{\psi_k}\right) \ast f\right)}_{\widetilde{A}^* f},
\end{equation}
where $\flip\left(\psi\right)\ix{x}\coloneqq\psi\ix{-x}$, and the over-line indicates the complex conjugate. Here we use a product-convolution scheme. 

Convolution interpolation schemes have been used in image restoration and deblurring \cite{FishEtAl96,NagyOleary98,TrussellHunt78} in photography \cite{TrussellFogel92}, astronomy \cite{Adorf94,FlickerRigaut05,RogersFiege11}, and microscopy \cite{PrezaConchello04}, as well as in wireless communication signal processing \cite{HrycakEtAl10}, ultrasound imaging \cite{NgPragerKingsbury07}, systems biology \cite{GiladVonHardenberg06}, and Hessian approximation in seismic inversion \cite{ZhuLiFomelEtAl16}.\footnote{In many of these applications, the impulse response is known as the point spread function (PSF), as it corresponds to the spreading of a point source of light as it passes through an optical system.} Aside from the application, convolution interpolation schemes differ based on how they construct the functions $\omega_k$ and $\psi_k$. For a comprehensive overview of existing schemes, we refer the reader to the summaries in \cite{DenisEtAl15, EscandeWeiss17, GentileCourbinMeylan13}. 

Existing schemes can be categorized based on whether the span of the functions $\omega_k$ is fixed, or the span of the functions $\psi_k$ is fixed, or both of the spans are fixed, or neither of the spans are fixed. Schemes then attempt to find the remaining (not fixed) functions and the coefficients for linear combinations of the fixed functions so that the error in the resulting operator approximation is small. Established choices for the span of the functions $\psi_k$ include the span of impulse responses of $A$ to point sources at a collection of fixed locations (we do this), subspaces of this span, and the span of functions with known analytic forms (e.g., Gaussians, spherical harmonics). Established choices for the span of the functions $\omega_k$ include spans of Fourier modes, piecewise polynomials  on a regular grid (e.g., piecewise constants, piecewise linear functions, B-splines), wavelets, radial basis functions \cite{BigotEscandeWeiss16}, and functions based on kriging. 

On one hand, existing schemes in which the functions $\omega_k$ are not fixed\footnote{The terminology for this is potentially confusing: in the literature, computed (rather than fixed) functions $\omega_k$ are known as ``adaptive'' weighting functions, but this is unrelated to our ``adaptive grid'' weighting functions.} require more access to $A$ than just the ability to apply it to vectors. On the other hand, existing schemes in which the functions $\omega_k$ are fixed do not permit spatial adaptivity, with one exception. This includes existing sectioning approaches that partition the domain into pieces on a regular grid, then use different functions $\psi_k$ for each piece \cite{NagyOleary98}. The exception is \cite{BardsleyEtAl06}, which, like this paper, proposes partitioning the domain with an adaptively refined grid. However, \cite{BardsleyEtAl06} only proposes the concept; they do not provide practical algorithms to perform the adaptivity.

In \cite{BelangerDemanet15}, matrix probing \cite{ChiuDemanet12} using basis matrices with $(y,x)$ entries that take the form $\omega_k\ix{x+y}\psi_k\ix{y-x}$ is used to approximate the exterior Dirichlet-to-Neumann map for a forward wave propagation problem. This approximation could be viewed as a middle ground between a product-convolution scheme and a convolution-product scheme, which would correspond to basis matrices of the form $\omega_k\ix{x}\psi_k\ix{y-x}$ and $\omega_k\ix{y}\psi_k\ix{y-x}$, respectively. After constructing the approximation, \cite{BelangerDemanet15} proposes converting it to $H$-matrix format for further use. Our approximation is different, but we also propose the same subsequent $H$-matrix conversion.

\subsubsection{Hierarchical matrices}
\label{sec:H_matrix_background}

Hierarchical matrices \cite{Hackbusch99} are matrices that may be full-rank, but the blocks of the matrix associated with clusters of degrees of freedom that are far away from each other (or satisfy some other admissibility condition) are low-rank. This structure allows for compressed storage and fast (nearly linear) matrix arithmetic, including matrix inversion and factorization. Special subclasses of $H$-matrices such as $H^2$-matrices \cite{HackbuschKhoromskijSauter00} (among others) allow for greater compression and faster matrix arithmetic. For an overview of $H$- and $H^2$-matrices, see \cite{Borm10,Hackbusch15}. 

Classical $H$-matrix construction techniques require access to the matrix entries of $A$, and hence are not applicable here. There exist matrix-free $H$-matrix construction techniques based on a recursive ``peeling-process'' \cite{LinYing11}, but these techniques have several subtle limitations. Although asymptotically scalable in theory, in practice the peeling process must apply the original operator to a large number of vectors. Furthermore, attempting to construct a less accurate approximation by applying the original operator to fewer vectors is not advisable (unlike our scheme where this is fine). Errors at any step of the peeling process compound during subsequent steps. Finally, the peeling process is purely algebraic. This makes the peeling process more general, at the cost of potentially being less efficient than specialized schemes (like ours) that take advantage of local translation-invariance or other properties of the operator being approximated.

\subsection{Setting and notation}
\label{sec:setting_notation}
We work in $l^2$ spaces on $\mathbb{Z}^d$ or subsets of $\mathbb{Z}^d$; these spaces arise when one discretizes a function on a continuous domain using a regular grid. Norms are denoted with $\nor{\cdot}$, or occasionally $\nor{\cdot}_{l^2(X)}$ if the domain is not clear from context. For linear operators we always use the Frobenius norm (square root of the sum of squares of all entries of the matrix representation of the linear operator). 

We routinely encounter Cartesian products of intervals, which we call \emph{boxes} and denote with a bold letter, as in $\inds{C}$. Boxes are characterized by their \emph{minimum point} and \emph{maximum point}: the points in the box that are component-wise less than or equal to all other points in the box, or greater than or equal to all other points in the box, respectively. We denote the minimum and maximum points of a box with the same letter as the box, but lower-case, and with the subscripts ``min'' and ``max'', respectively. For example, $\inds{C} = \bigtimes_{i=1}^d [c_\text{min}^i, c_\text{max}^i]$, where $\bigtimes$ is the Cartesian product of sets. We write $\corners(\inds{C})\coloneqq\bigtimes_{i=1}^d \{c_\text{min}^i, c_\text{max}^i\}$ to denote the set of corners of $\inds{C}$. The (approximate) midpoint, $c_\text{mid}$, of the box $\inds{C}$ is the integer vector closest to the real vector $(c_\text{max} + c_\text{min})/2$. The \emph{linear dimension} of a box is the sum of all the dimensions of the box: $\sum_{i=1}^d c_\text{max}^i - c_\text{min}^i$. 

Minkowski set arithmetic is used for addition and subtraction of one set with another set, negation of a set, and addition and subtraction of a set with a point:
\begin{equation*}
X + Y = \{x+y: x \in X, y \in Y\}, \quad\quad X - Y = \{x-y: x \in X, y \in Y\},
\end{equation*}
and similar for negation of a set, and addition and subtraction of a point from a set. The number of elements in a set $X$ is denoted $|X|$. We reserve $N$ for the total number of points in the domain: $N\coloneqq |\inds{\Omega}|$.

The evaluation of $f$ at $x$ is denoted $f\ix{x}$, and $f\ix{\inds{C}} \in l^2(\inds{C}-c_\text{min})$, with $\left(f\ix{\inds{C}}\right)\ix{x} \coloneqq f\ix{x + c_\text{min}}$. Likewise, $M\ix{y,x}$ is the $(y,x)$ ``matrix entry'' of $M$, and $M\ix{\inds{T}, \inds{S}} \in l^2\left((\inds{T}-t_\text{min}) \times (\inds{S}-s_\text{min})\right)$ with $\left(M\ix{\inds{T}, \inds{S}}\right)\ix{y,x} \coloneqq M\ix{y + t_\text{min}, x+s_\text{min}}$. That is, $M\ix{\inds{T},\inds{S}}$ is the $\inds{T},\inds{S}$ ``block'' of $M$. A dot within indexing brackets, as in $M\ix{\inds{C},~\cdot~}$ or $M\ix{~\cdot~,\inds{C}}$, indicates the matrix of all columns or rows of $M$ corresponding to points in $\inds{C}$, respectively. The action of a linear operator $M$ on a vector $f$ is denoted $Mf$. We write $M^*$ to denote the adjoint of $M$. That is, $M^*\ix{y,x} = \overline{M\ix{x,y}}$, where the over-line indicates the complex conjugate.

A dot between two functions denotes \emph{pointwise multiplication} of those functions: 
\begin{equation*}
\left(f \cdot g\right)\ix{x} \coloneqq f\ix{x} g\ix{x}.
\end{equation*}
An asterisk between two functions denotes \emph{convolution} of those functions:
\begin{equation}
\label{eq:convolution_definition}
(\psi \ast f)\ix{y} \coloneqq \sum_{x \in \mathbb{Z}^d} f\ix{x} \psi\ix{y-x}.
\end{equation}
If the domains of functions $f,\psi$ are only subsets of $\mathbb{Z}^d$, we define their convolution to be the result of extending $f,\psi$ by zero so that they are defined on all of $\mathbb{Z}^d$, then convolving them using formula \eqref{eq:convolution_definition}. We use the term ``convolution rank'' to denote the number of terms in a weighted sum of convolution operators (e.g., $r$ in \eqref{eq:conv_op_approx}).

We define the functions
\begin{equation*}
\delta_p\ix{x}:=\begin{cases}
1, & x=p, \\
0, & \text{otherwise}
\end{cases} \quad \text{and} \quad \mathbbm{1}_X := \begin{cases}
1, & x\in X, \\
0, & \text{otherwise}.
\end{cases}
\end{equation*}
We denote the support of a function $f$ by $\supp(f)$. By the ``support'' of a function, we mean the largest set on which the function could, in principle, be non-zero (independent of whether the numerical value of the function happens to be zero). We call a function of $N$ \emph{nearly linear} if it scales as $O(N \log^a N)$ for $N \rightarrow \infty$, where $a$ is some small non-negative integer (say $a \in \{0,1,2\}$).

\section{The adaptive product-convolution approximation}
\label{sec:convolution_operator_approx}

As discussed in Section \ref{sec:intro0}, if $A$ were translation-invariant (i.e., if \eqref{eq:translation_invariance} held with equality for all $x,y \in \mathbb{Z}^d$), then $A$ would be the convolution operator defined by the action $Af = \varphi_p \ast f$, where $\varphi_p$ is the impulse response of $A$ at $p$, as defined in \eqref{eq:varphi_p}. For example, the solution operator for a homogeneous PDE on an unbounded domain is translation-invariant, and $\varphi_p$ is the Green's function for the PDE. Of course, translation-invariant operators are rare in practice. It is more common for $A$ to only be \emph{approximately} translation-invariant (see Figure \ref{fig:translation_invariance}), and for the approximate translation-invariance to be valid only \emph{locally}. That is, 
\begin{equation}
\label{eq:translation_fail3}
A\ix{p+y-x,p} \approx A\ix{y,x} \quad\text{when }x\in U
\end{equation}
for some neighborhood $U$ consisting of points ``near'' $p$. We will provide a rigorous analysis of approximation errors in Section \ref{sec:error_analysis}; for now we leave the exact nature of this approximate equality ($\approx$) intentionally vague. Just as translation-invariance of $A$ implies that $A$ is a convolution operator, local approximate translation-invariance of $A$ implies that $A$ can be locally approximated by a convolution operator. Specifically, \eqref{eq:translation_fail3} implies
\begin{equation}
\label{eq:conv_approx_equal}
Ag \approx \varphi_p \ast g \quad \text{when~}\supp(g) \subset U.
\end{equation}
In order to approximate the action of $A$ on functions $f$ supported on a larger region of interest, we patch together local convolution operator approximations. Let $\{U_k\}_{k=1}^r$ be a collection of sets covering $\supp(f)$, let $\{w_k\}_{k=1}^r$ be a partition of unity subordinate to this cover, let $p_k \in U_k$ for $k=1,\dots,r$, and define $\varphi_k \coloneqq \varphi_{p_k}$. If the following local approximations hold:
\begin{equation}
\label{eq:local_conv_approx5}
Ag \approx \varphi_k \ast g \quad \text{when }\supp(g) \subset U_k, \quad k=1,\dots,r,
\end{equation}
then $A$ can be globally approximated as follows:
\begin{equation}
Af = A \sum_{k=1}^r w_k \cdot f 
= \sum_{k=1}^r A (w_k \cdot f) 
\approx \sum_{k=1}^r \varphi_k \ast (w_k \cdot f). \label{eq:unextended_conv_scheme}
\end{equation}
The first equality follows from the partition unity property of the functions $w_k$, the second follows from the linearity of $A$, and the approximate equality follows from the local approximation property \eqref{eq:local_conv_approx5} and the fact that $\supp(w_k \cdot f) \subset U_k$.

\subsection{Overview of the approximation}
\label{sec:convolution_product_format}

The previous derivation leads us to approximate $A$ with the following \emph{product-convolution approximation}:
\begin{equation}
\label{eq:conv_op_approx_main}
\widetilde{A} f \coloneqq \sum_{k=1}^r \varphi_k^E \ast (w_k \cdot f),
\end{equation}
where
\begin{itemize}
\item $\left\{\varphi_k^E\right\}_{k=1}^r$ are modified (``extended'') versions of the impulse responses
\begin{equation}
\label{eq:varphi_p_section2}
\varphi_k\ix{z} = \left(A \delta_{p_k}\right)\ix{z+p_k}, \quad z \in \inds{\Omega}-p_k,
\end{equation}
for a collection of \emph{sample points} $\{p_k\}_{k=1}^r$.
\item The sample points $\{p_k\}_{k=1}^r$ reside in a collection of overlapping sets $\{U_k\}_{k=1}^r$ that cover $\inds{\Omega}$: 
\begin{equation*}
p_k \in U_k \text{ for }k=1,\dots,r \quad \text{and}\quad \inds{\Omega} \subset \bigcup_{k=1}^r U_k.
\end{equation*}
\item $\left\{w_k\right\}_{k=1}^r$ is a partition of unity subordinate to the cover:
\begin{equation*}
\supp(w_k) \subset U_k \text{ for }k=1,\dots,r \quad \text{and}\quad \sum_{k=1}^r w_k\ix{x} = 1 \quad \text{for all }x \in \inds{\Omega}.
\end{equation*}
\end{itemize}
Our scheme is defined by the points $p_k$, the sets $U_k$, the partition of unity weighting functions $w_k$, and the extended impulse response functions $\varphi_k^E$. 

In general, translation-invariance varies spatially. By this, we mean that the size of the neighborhood $U$ on which the error in \eqref{eq:translation_fail3} is sufficiently small depends on the location of $U$. To fix ideas, suppose that $A$ is the solution operator for an inhomogeneous elliptic PDE. In this case, the size of $U$ will typically be small if the coefficient in the PDE varies over short length scales within $U$, and large if the coefficient varies over large length scales within $U$. In order to capture such spatial variations in translation-invariance while minimizing the number of sample points used, we choose $p_k$ and $U_k$ adaptively (Sections \ref{sec:adaptive_sample_points} and \ref{sec:adaptive_algorithm}). A randomized adjoint based a-posteriori error estimator (Section \ref{sec:randomized_estimator}) drives the adaptivity. 

Due to boundary effects, translation-invariance typically fails in directions perpendicular to a boundary, but holds in directions parallel to that boundary. For example, let $\varphi_p$ be the Green's function at $p$ for a homogeneous PDE on an infinite half-space. Although $\varphi_p$ changes as $p$ approaches the boundary, by symmetry it does not change as $p$ moves parallel to the boundary. In order to address this direction-dependent translation-invariance, we refine anisotropically, subdividing preferentially in directions that $\varphi_p$ changes the most as a function of $p$ (Section \ref{sec:anisotropic}).

The adaptive refinement procedure creates unusually shaped neighborhoods $U_k$. We construct harmonic weighting functions, $w_k$, on these sets by solving local Laplace problems (Section \ref{sec:poisson_weighting_functions}).

\begin{figure}
\begin{subfigure}[t]{0.28\textwidth}
	\center
    \includegraphics[scale=0.20]{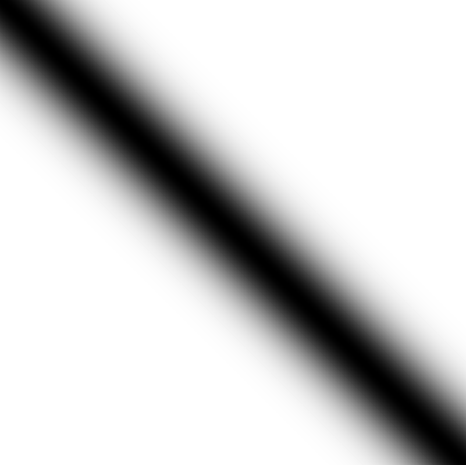}
    \caption{$A$}
    \label{fig:1d_conv_exact}
\end{subfigure}
\begin{subfigure}[t]{0.28\textwidth}
	\center
    \includegraphics[scale=0.20]{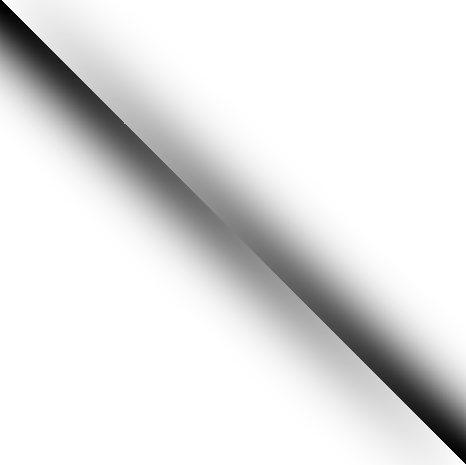}
    \caption{$\widetilde{A}$}
    \label{fig:1d_conv_approx}
\end{subfigure}
\begin{subfigure}[t]{0.4\textwidth}
	\center
    \includegraphics[scale=0.20]{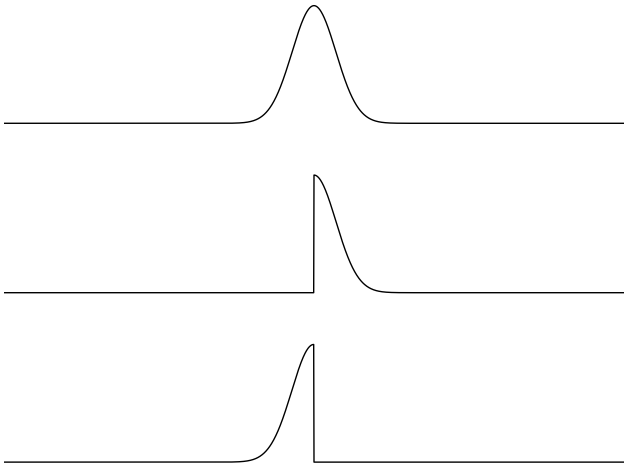}
    \caption{Top: $\varphi$. Mid: $\varphi_\text{left}$. Bot: $\varphi_\text{right}$.}
    \label{fig:1d_conv_varphis}
\end{subfigure}
\caption{Extending impulse responses by zero leads to boundary artifacts even if $A$ is, itself, a convolution operator. Here $A$ (\ref{fig:1d_conv_exact}) takes a function defined on $[1,N]$, extends it by zero to $\mathbb{Z}$, convolves it with a Gaussian $\varphi$ (\ref{fig:1d_conv_varphis}), then restricts the result to $[1,N]$. The approximation, $\widetilde{A}$  (\ref{fig:1d_conv_approx}), linearly interpolates between convolution with $\varphi_\text{left}$ at $1$ and $\varphi_\text{right}$ at $N$, where $\varphi_\text{left}$ and $\varphi_\text{right}$ (\ref{fig:1d_conv_varphis}) are the impulse responses of $A$ to point sources centered at $1$ and $N$, respectively, with extension by zero used as needed. Black indicates value $1$, and white indicates value $0$ in \ref{fig:1d_conv_exact} and \ref{fig:1d_conv_approx}.}
\label{fig:1d_conv}
\end{figure}

Because of boundaries, the domains of definition of the functions $\varphi_k$ are not large enough for the convolutions in the naive product-convolution formula, $\sum_{k=1}^r \varphi_k \ast \left(w_k \cdot f \right)$, to be well-defined. Extending functions by zero as needed makes these convolutions well-defined, but this leads to boundary artifacts wherever zeros are used in place of undefined data. These boundary artifacts are purely a side effect of the scheme and are unrelated to real boundary effects present in the underlying operator $A$; they occur even in the case where $A$ is, itself, a convolution operator (see Figure \ref{fig:1d_conv}). To eliminate such boundary artifacts, we extend the functions $\varphi_k$ outside of their natural support by using information from neighboring functions $\varphi_j$ to create ``extended'' impulse response functions $\varphi_k^E$ (Section \ref{sec:boundary_extension}).

\subsection{Adaptive grid structure}
\label{sec:adaptive_sample_points}

We will choose the sample points, $p_k$, so that they form an adaptively refined rectilinear grid (for example, see Figure \ref{fig:adaptive_subgrid}). This section defines the structure of the adaptive grid; the procedure for constructing it will be explained in Section \ref{sec:adaptive_algorithm}. 

\begin{figure}
\begin{subfigure}[b]{0.45\textwidth}
	\center
    \includegraphics[scale=0.7]{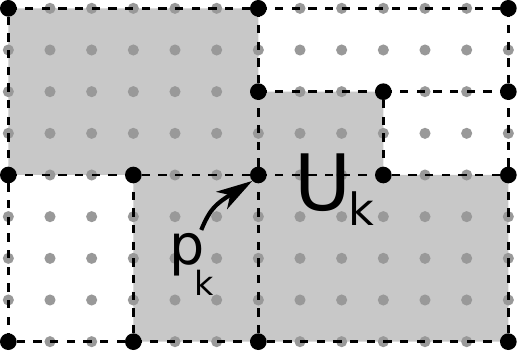}
    \caption{Blocky neighborhood $U_k$ associated with an interior sample point $p_k$.}
    \label{fig:adaptive_grid_1}
\end{subfigure}
\hfill
\begin{subfigure}[b]{0.45\textwidth}
	\center
    \includegraphics[scale=0.7]{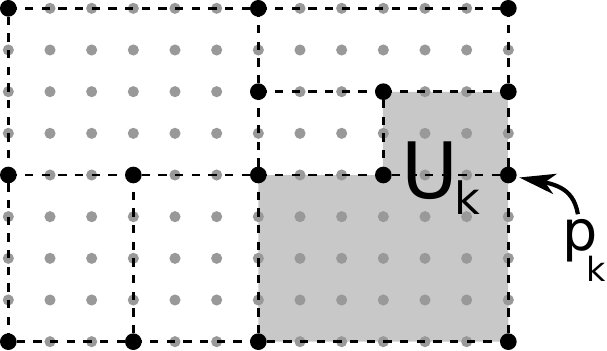}
    \caption{Blocky neighborhood $U_k$ associated with a boundary sample point $p_k$.}
    \label{fig:adaptive_grid_2}
\end{subfigure}
\caption{Sample points $p_k$ (black points) form an adaptively refined grid within $\inds{\inds{\Omega}}$ (all gray and black points). The blocky neighborhood $U_k$ associated with sample point $p_k$ (shaded light gray region) is the union of all leaf cells that contain $p_k$.}
\label{fig:adaptive_subgrid}
\end{figure}

We organize the domain $\inds{\Omega}$ into a binary tree, $\mathcal{T}$, of boxes $\inds{C}\subset \inds{\Omega}$ which we call \emph{cells}. The root of $\mathcal{T}$ is the whole domain $\inds{\Omega}$. Cells may be either refined or not refined; refined cells are internal nodes in $\mathcal{T}$ and unrefined cells are leaves of $\mathcal{T}$. We denote the set of all leaves of the tree by $\leaves(\mathcal{T})$. Refined cells $\inds{C}$ are subdivided in a chosen direction into a set of two child cells that share an internal facet (more about how we choose the subdivision direction in Section \ref{sec:anisotropic}). We denote the set of children of $\inds{C}$ by $\children(\inds{C})$. The corners of all cells form the set of sample points:
\begin{equation*}
\{p_k\}_{k=1}^r = \bigcup_{\inds{C} \in \mathcal{T}} \corners(\inds{C}).
\end{equation*}
Since the cells share facets, typically more than one cell contains a given sample point. We write 
\begin{equation*}
\cells(p_k)\coloneqq\{\inds{C}: \inds{C} \in \leaves(\mathcal{T}), p_k \in \inds{C}\}
\end{equation*}
to denote the set of all leaf cells containing $p_k$. We define the \emph{blocky neighborhood}, $U_k$, associated with a sample point $p_k$ as the union of all leaf cells containing $p_k$:
\begin{equation*}
U_k \coloneqq \bigcup_{\inds{C}_i \in \cells(p_k)} \inds{C}_i.
\end{equation*}
Sample points $p_k$ and $p_j$ are \emph{neighbors} if they share a common leaf cell. That is, there exists a leaf cell $\inds{C}$ such that $p_k \in \inds{C}$ and $p_j \in \inds{C}$. Note that under this definition $p_k$ is neighbors with itself. We write $\nbrs(k) \subset \{1,\dots,r\}$ to denote the set of indices of sample points that are neighbors of $p_k$, including $p_k$ itself. In other words, $j\in\nbrs(k)$ if $p_k$ and $p_j$ are neighbors.

\subsection{Adaptive refinement algorithm}
\label{sec:adaptive_algorithm}
Starting with $\inds{\Omega}$ subdivided once in all directions, we repeatedly estimate the error in all cells in $\leaves\left(\mathcal{T}\right)$ using an a-posteriori error estimator, then refine the leaf cell with the largest error. The refinement process continues until either (a) the desired error in the approximation is achieved, or (b) a predetermined maximum number of sample points $p_k$ is reached.  At each step of the refinement process we construct or modify the functions $w_k$ and $\varphi_k^E$ using methods that will be described in Sections \ref{sec:poisson_weighting_functions}, \ref{sec:boundary_extension}, and \ref{sec:construction_cost}. We perform the a-posteriori error estimation with a randomized method that will be described in Section \ref{sec:randomized_estimator}. We choose which direction to subdivide cells in using a method that will be described in Section \ref{sec:anisotropic}. The complete algorithm is summarized in Algorithm \ref{alg:construction}.

\subsection{Harmonic weighting functions}
\label{sec:poisson_weighting_functions}

\begin{figure}
	\center
    \includegraphics[scale=0.2]{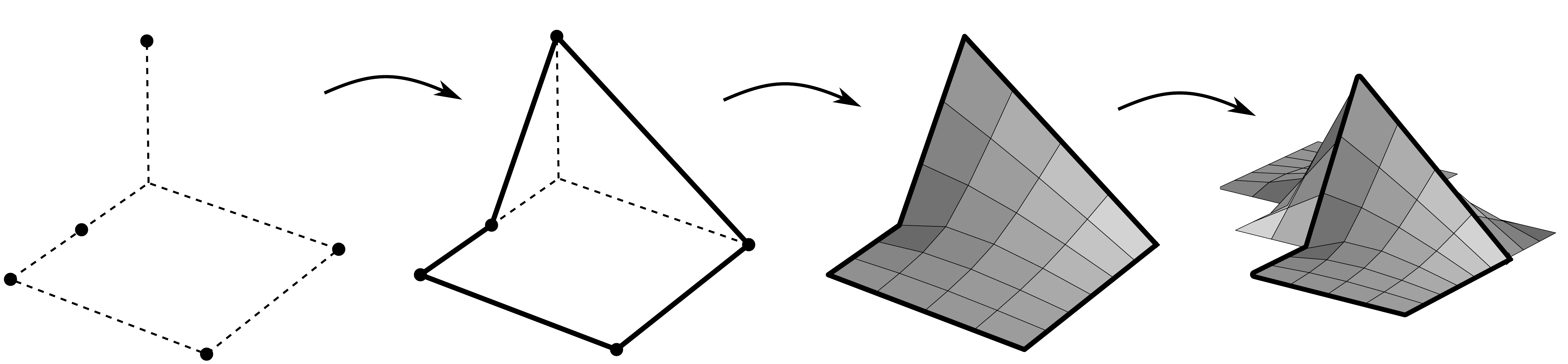}
    \caption{Construction of $w_k$ for $d=2$. For each box in $U_k$ we assign $w_k$ the value $1$ at sample point $p_k$ and $0$ at all other sample points. For edges between sample points, we compute the values of $w_k$ by solving the discrete 1-D Laplace equation, using the previously assigned values at sample points as Dirichlet boundary conditions. For faces, we compute the values of $w_k$ by solving the discrete 2-D Laplace equation, using the previously computed edge values as Dirichlet boundary conditions. Finally, we form $w_k$ on $U_k$ by combining its constituent pieces on each box.}
    \label{fig:harmonic}
\end{figure}

We construct harmonic partition of unity weighting functions, $w_k$, by solving discrete local Laplace (diffusion) problems recursively on subsets of $U_k$. This process is equivalent to the construction of harmonic basis functions in finite element methods \cite{Bishop14}, and also shares conceptual ties with partition of unity finite element methods \cite{BabuskaMelenk95} and the construction of coarse basis functions in agglomerated element algebraic multigrid \cite{JonesVassilevski01}. 

The blocky neighborhood $U_k$ is a union of $d$-dimensional boxes. The boundary of each $d$-dimensional box is a union of $(d-1)$-dimensional facets, each of which is a box. There are $2d$ facets, corresponding to either the front or the back of the box in each coordinate direction. Facets that contain hanging nodes (``broken facets'') are the union of several smaller $(d-1)$-dimensional boxes. Hence the boundary of each $d$-dimensional box can be expressed as the union of $(d-1)$-dimensional boxes, where we exclude broken facets in favor of their constituent smaller boxes. In the same way, the boundary of each $(d-1)$-dimensional box is a union of $(d-2)$-dimensional boxes, and so forth all the way down until we reach a set of $0$-dimensional sample points. We build harmonic weighting functions by solving the Laplace equation ($-\Delta w_k = 0$) on these boxes recursively in dimension, using the values from lower-dimensional boxes as Dirichlet boundary conditions for higher-dimensional boxes. For sample points $p_j$ (the lowest level), we assign $w_k\ix{p_k}=1$ and $w_k\ix{p_j}=0$ for $j \neq k$. Figure \ref{fig:harmonic} illustrates this process for $d=2$. Linearity, the maximum principle, and induction on boxes of increasing dimension show that the functions $w_k$ form a partition of unity on $\inds{\Omega}$.

For the discrete Laplace equation we use the (positive definite) discrete graph Laplacian; this is equivalent to discretizing the continuous Laplacian using a standard Kronecker sum finite difference approximation on a regular grid. The local Laplace problems can be solved efficiently (in time proportional to the number of unknowns) with multigrid \cite{BankDupont81, BraessHackbusch83}.

\subsection{Extended impulse response functions}
\label{sec:boundary_extension}

To construct $\varphi_k^E$, we first compute the impulse responses $\varphi_k$ of $A$ at the points $p_k$ by applying $A$ to point sources, then translating the results (see \eqref{eq:varphi_p_section2}). To eliminate boundary artifacts, we create $\varphi_k^E$ by extending the support of $\varphi_k$, using data from neighboring functions $\varphi_j$ to fill in regions outside of $\supp(\varphi_k)$. 
\begin{enumerate}
\item For $z$ within $\supp(\varphi_k)$, we set $\varphi_k^E\ix{z}\coloneqq\varphi_k\ix{z}$. 
\item For $z$ outside $\supp(\varphi_k)$ but within $\supp(\varphi_j)$ for at least one neighboring $\varphi_j$, we define $\varphi_k^E\ix{z}$ as the average of all neighboring $\varphi_j\ix{z}$ whose support contains $z$.
\item For $z$ outside $\supp(\varphi_k)$ and outside $\supp(\varphi_j)$ for all neighboring $\varphi_j$, we set $\varphi_k^E\ix{z}\coloneqq 0$.
\end{enumerate}
Figure \ref{fig:neighbor_extension} illustrates this procedure for a $1$-dimensional example. Our theory still holds if we use any weighted average of neighboring $\varphi_j\ix{z}$ in Step 2, provided the weights are non-negative and sum to one. We use the average since it simplifies the implementation and the explanation, and since more elaborate schemes are likely to yield only minimal improvements. The fact that we set some entries of $\varphi_k^E\ix{z}$ to zero in Step 3 is irrelevant since our scheme never accesses these entries (this will follow from Proposition \ref{prop:weighting_product_space}).

\begin{figure}
\begin{subfigure}[t]{0.24\textwidth}
	\center
    \includegraphics[scale=0.35]{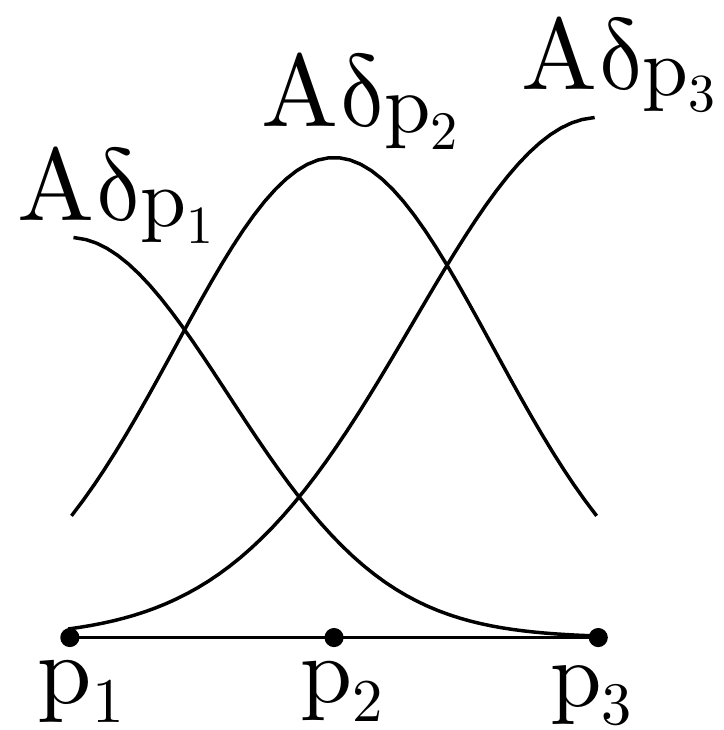}
    \label{fig:neighbor_extension3}
\end{subfigure}
\begin{subfigure}[t]{0.37\textwidth}
	\center
    \includegraphics[scale=0.35]{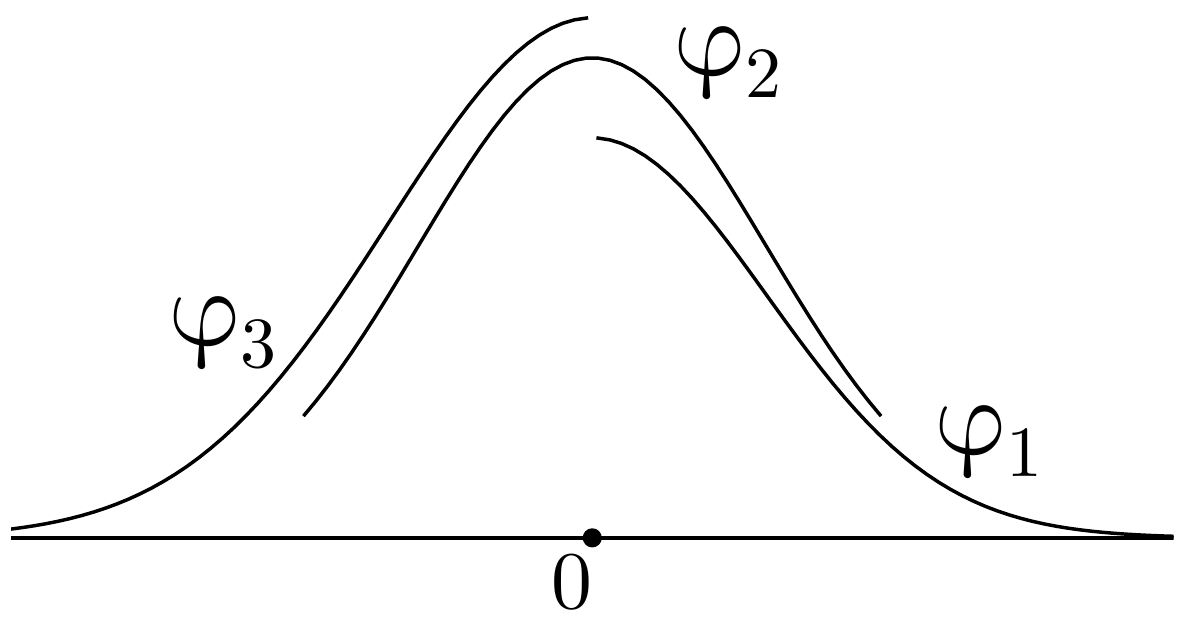}
    \label{fig:neighbor_extension2}
\end{subfigure}
\begin{subfigure}[t]{0.37\textwidth}
	\center
    \includegraphics[scale=0.35]{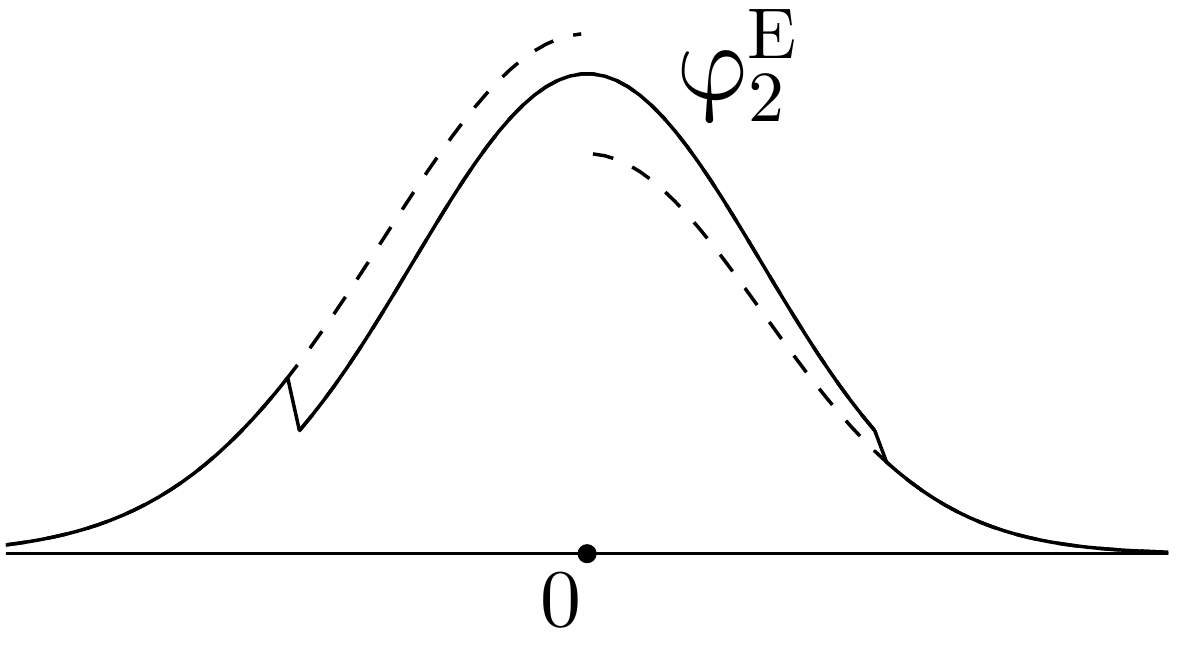}
    \label{fig:neighbor_extension1}
\end{subfigure}
\caption{Illustration of impulse response extension procedure in $1$ dimension. To construct $\varphi_2^E$, we extend the support of $\varphi_2$ by filling in regions where $\varphi_2$ is undefined with values from $\varphi_1$ and $\varphi_3$.}
\label{fig:neighbor_extension}
\end{figure}

In preparation for the theory in Section \ref{sec:error_analysis}, we now describe the process of constructing $\varphi_k^E$ more precisely. First, we construct the following counting functions:
\begin{equation*}
c_k \coloneqq \mathbbm{1}_{\inds{\Omega}-p_k} + \sum_{\substack{j\in\nbrs(k)\\j\neq k}} \mathbbm{1}_{\left(\inds{\Omega}-p_j\right) \setminus (\inds{\Omega}-p_k)}.
\end{equation*}
Since $\supp(\varphi_j) = \inds{\Omega}-p_j$, $c_k\ix{z}$ counts how many $\varphi_j$ will contribute to $\varphi_k^E\ix{z}$. Next we compute 
\begin{equation*}
v_k\ix{z} \coloneqq \begin{cases}
	1/c_k\ix{z}, & z \in \supp(c_k) \\
	0, & \text{otherwise},
	\end{cases}
\end{equation*}
and define
\begin{equation}
\label{eq:vj_def}
v_k^{(j)} \coloneqq \begin{cases}
	v_k \cdot \mathbbm{1}_{\inds{\Omega}-p_k}, & j=k \\
	v_k \cdot \mathbbm{1}_{\left(\inds{\Omega} - p_j\right) \setminus \left(\inds{\Omega}-p_k\right)}, & j\in\nbrs(k),~j \neq k,
	\end{cases}
\end{equation}
The function $v_k^{(j)}\ix{z}$ is the weight given to neighboring impulse response $\varphi_j$ at point $z$ when constructing $\varphi_k^E$. Finally, we construct $\varphi_k^E$:
\begin{equation}
\label{eq:def_varphiE}
\varphi_k^E \coloneqq \sum\limits_{j\in\nbrs(k)} v_k^{(j)} \cdot  \varphi_j.
\end{equation}

\subsection{Randomized a-posteriori error estimator}
\label{sec:randomized_estimator}

In order to decide which cells to refine, we wish to compute the error in the approximation,
\begin{equation}
\label{eq:relative_error}
e_\inds{C} := \nor{\left(\widetilde{A} - A\right)\ix{\inds{\Omega}, \inds{C}}},
\end{equation}
for all cells $\inds{C} \in \leaves(T)$. Computing these norms is prohibitively expensive, so instead we estimate them. If $M$ is any matrix with $N$ columns, then the following sample average approximation estimates the square of its Frobenius norm:
\begin{equation}
\label{eq:saa_XZ}
||M||^2 = \mathbb{E}\left( ||M \zeta||^2 \right) \approx \frac{1}{\nadj}\sum_{i=1}^{\nadj} ||M \zeta_i||^2 = \frac{1}{\nadj}\nor{M Z}^2,
\end{equation}
where $\zeta, \zeta_i \sim N(0,1)^N$, are independent and identically distributed (i.i.d.) Gaussian random vectors, $\mathbb{E}$ is the expected value, $\nadj$ is the number of samples used in the sample average approximation, and $Z \sim N(0,1)^{N \times \nadj}$ is an i.i.d. Gaussian random matrix (the matrix with columns $\zeta_i$) \cite{AvronToledo11}. Hence we can form an estimator, $\eta_\inds{C} \approx e_\inds{C}$, by forming a random matrix $Z \sim N(0,1)^{N \times \nadj}$, computing
\begin{equation*}
Y = A^* Z \quad \text{and} \quad \widetilde{Y} = \widetilde{A}^* Z,
\end{equation*}
then extracting blocks of the results, and taking norms:
\begin{equation}
\label{eqn:randomized_estimator_ec}
\eta_\inds{C} \coloneqq \frac{1}{\sqrt{\nadj}}\|\widetilde{Y}\ix{\inds{C},~\cdot~} - Y\ix{\inds{C},~\cdot~}\|.
\end{equation}
By performing the randomized sample average approximation with the adjoints $A^*$ and $\widetilde{A}^*$, we apply these operators once per sample, then post process the results to get estimators for all cells. Using the original operators $A$ and $\widetilde{A}$ instead would require us to apply these operators to new random vectors for each cell.

It is straightforward to adapt the Chernoff bound in \cite{AvronToledo11} to get an upper bound on the number of samples required. However, this bound is overly pessimistic; in practice we find the estimator is effective with only a handful of samples.

\subsection{Anisotropic refinement: choosing the subdivision direction}
\label{sec:anisotropic}
We refine anisotropically by estimating the direction that $\varphi_p$ changes the most as a function of $p$, then subdividing in that direction. This allows us to capture changes to $\varphi_p$ in directions where translation-invariance fails, without refining the grid in directions where translation-invariance holds.

Let $\inds{C}$ be a cell that we have chosen to subdivide based on the randomized a-posteriori error estimator described in Section \ref{sec:randomized_estimator}. For each coordinate direction $i$ in which $\inds{C}$ is big enough to be refined ($c_\text{max}^i - c_\text{min}^i > 2$), we partition the functions $\varphi_k^E$ associated with the corners of $\inds{C}$ into two groups. One group is the set of $\varphi_k^E$ associated with corners in the ``front'' of the cell ($+$) in the $i$th coordinate direction, and the other group is the set of $\varphi_k^E$ associated with the ``back'' of the cell ($-$) in the $i$th coordinate direction:
\begin{align*}
\Psi^{i+} \coloneqq & \{\varphi_k^E : p_k \in \corners(\inds{C}),~ p_k^i = c_\text{max}^i \}, \\
\Psi^{i-} \coloneqq & \{\varphi_k^E : p_k \in \corners(\inds{C}),~ p_k^i = c_\text{min}^i \}.
\end{align*}
Next, we construct ``average'' $\varphi_k^E$ functions for the front and back of the cell, respectively:
\begin{equation*}
\varphi^{i+} \coloneqq \frac{1}{2^{d-1}}\sum_{\varphi^E \in \Psi^{i+}} \varphi^E
\quad\text{and}\quad
\varphi^{i-} \coloneqq \frac{1}{2^{d-1}}\sum_{\varphi^E \in \Psi^{i-}} \varphi^E.
\end{equation*}
Then we determine how much these average impulse responses change from the front to the back in direction $i$ by computing $\nor{\varphi^{i+} - \varphi^{i-}}_{l^2(\inds{\Omega}-c_\text{mid})}$. Finally, we subdivide $\inds{C}$ in the coordinate direction $i$ in which the average impulse response changes the most.

\subsection{Construction cost}
\label{sec:construction_cost}

\begin{algorithm}
	\renewcommand{\algorithmicrequire}{\textbf{Input:}}
	\renewcommand{\algorithmicensure}{\textbf{Output:}}
	\caption{Construction of $\widetilde{A}$}
	\label{alg:construction}
	\begin{algorithmic}[1]
		\Require $v \mapsto A v$, $v \mapsto A^*v$, $\inds{\Omega}$, $\tau$, $q$
		\Ensure $\left(w_k, \varphi_k^E\right)_{k=1}^r$
		\Statex
		\State Draw random matrix $Z \sim N(0,1)^{N \times q}$
		\State Compute $Y = A^* Z$\Comment{Cost: $q$ applications of $A^*$}
		\State Initialize $\mathcal{T}$ with $\inds{\Omega}$ as its root
		\State Refine $\mathcal{T}$ by subdividing $\inds{\Omega}$ once in each coordinate direction
		\State Construct blocky neighborhoods $U_k$
		\State Construct harmonic weighting function $w_k$
		\State Compute impulse response functions $\varphi_k = A \delta_{p_k}$\Comment{Cost: $3^d$ applications of $A$}
		\State Construct extended impulse response functions $\varphi_k^E$
		\State Compute $\widetilde{Y} = \widetilde{A}^* Z$\Comment{Cost: $q \times 3^d$ convolutions}
		\State Form local error estimators $\eta_\inds{C}$
		\State Form overall error estimator $\eta_\inds{\Omega}$
		\While{$\eta_\inds{\Omega} > \tau$}
			\State Find cell $\inds{C} \in \leaves{\mathcal{T}}$ with the largest $\eta_\inds{C}$
			\State Determine subdivision direction, $i$, for $\inds{C}$
			\State Subdivide $\inds{C}$ in direction $i$
			\State Construct $U_k$ that are new or modified by the refinement
			\State Construct $w_k$ for new or modified $U_k$
			\State Compute $\varphi_k = A \delta_{p_k}$ for all new $p_k$\Comment{Cost: $1$ application of $A$ per new $p_k$}
			\State Construct new or modified $\varphi_k^E$
			\State Update $\widetilde{Y}$\Comment{Cost: $O(q)$ convolutions per new $p_k$}
			\State Form new or modified local error estimators $\eta_{\inds{C}}$
			\State Form overall error estimator $\eta_\inds{\Omega}$
		\EndWhile
	\end{algorithmic}
\end{algorithm}

Algorithm \ref{alg:construction} shows the complete algorithm for constructing $\widetilde{A}$. Updating $\widetilde{A}$ after refining a cell requires us to apply $A$ to point sources centered at the new sample points created during the refinement. Hence the entire refinement process requires us to apply $A$ to $r$ vectors, where $r$ is the total number of sample points in the final product-convolution approximation.

The dominant cost in the error estimation process is the cost of computing $A^*Z$ and $\widetilde{A}^*Z$ for a random matrix $Z$ with $q$ columns. Since $A^*Z$ is constant throughout the refinement process, we compute it once at the beginning. 

Although $\widetilde{A}$ changes every time we refine a cell, after performing a refinement we do not have to recompute $\widetilde{A}^*Z$ from scratch. To see this, recall from \eqref{eq:adjoint_random_z} that the adjoint of our product-convolution operator is a convolution-product operator with the convolution functions reflected about the origin and complex conjugated. In order to recompute $\widetilde{A}^*Z$ after refining cells, we only need to compute the convolutions $\flip\left(\overline{\varphi_k^E}\right) \ast \zeta_i$ for each column, $\zeta_i$, in $Z$, and each sample point, $p_k$, that is \emph{new} or has a \emph{new neighbor}.\footnote{The function $\varphi_k^E$ depends on neighboring impulse responses due to the extension procedure.} The convolutions for old sample points without new neighbors have been computed previously and can be re-used within \eqref{eq:adjoint_random_z}. Thus the error estimation process requires computing $O(r \nadj)$ convolutions. As we will discuss in Section \ref{sec:apply_H_to_vectors}, each of these convolutions can be done with the FFT in $O(N \log N)$ work. Updating the functions $w_k$ can be done locally. This requires negligible work compared to the other costs already discussed. Putting all these pieces together, constructing $\widetilde{A}$ requires
\begin{equation*}
O\left(r C + \nadj C_* + r \nadj N \log N\right),
\end{equation*}
work, where $C$ and $C_*$ are the costs to apply $A$ and $A^*$ to one vector, respectively.

\section{Using the product-convolution approximation}
\label{sec:H_matrix}
The product-convolution format allows us to perform useful operations with $\widetilde{A}$ that we cannot perform with $A$.

\subsection{Computing matrix entries of $\widetilde{A}$} 
\label{sec:matrix_entries}
Our approximation $\widetilde{A}$ is a product-convolution scheme and therefore (as seen in \eqref{eq:conv_prod_vs_prod_conv_entries}) has the following matrix entries:
\begin{equation}
\label{eq:conv_matrix_entries}
\widetilde{A}\ix{y,x} = \sum_{k=1}^r w_k\ix{x} \varphi_k^E\ix{y-x} = \sum_{k:x \in U_k} w_k\ix{x} \varphi_k^E\ix{y-x}.
\end{equation}
Using \eqref{eq:conv_matrix_entries} we can compute individual matrix entries of $\widetilde{A}$ in $O(1)$ time even though $\widetilde{A}$ is not stored in memory in the conventional sense.

\subsection{Applying $\widetilde{A}$ or $\widetilde{A}^*$ to vectors} 
\label{sec:apply_H_to_vectors}
Applying $\widetilde{A}$ or $\widetilde{A}^*$ to a vector requires computing $r$ convolutions, $r$ pointwise vector multiplications, and some vector additions (see equations \eqref{eq:conv_op_approx_main} or \eqref{eq:adjoint_random_z}, respectively). Out of these operations, the $r$ convolutions are the most computationally expensive. Since the convolution theorem allows us to compute each of these convolutions using the FFT (after appropriate zero padding)  \cite{HirschEtAl10} at $O(N \log N)$ cost, the cost of applying $\widetilde{A}$ or $\widetilde{A}^*$ to a vector is $O(r N \log N)$.

\subsection{Applying blocks of $\widetilde{A}$ or $\widetilde{A}^*$ to vectors} 
\label{sec:subblock_conv_fast}

One can implicitly apply a convolution operator to a function that is supported in a source box $\inds{S}$ then restrict the results to another target box $\inds{T}$, by performing a convolution between a function supported on a box with the same shape as $\inds{S}$ and a function supported on a box with the same shape as $\inds{T}-\inds{S}$, then translating the results. Specifically, a change of variables shows that if $f$ is supported on $\inds{S}$, then
\begin{equation*}
(\varphi \ast f)\ix{\inds{T}} = (\varphi_0 \ast f_0)\ix{\inds{T}'},
\end{equation*}
where 
\begin{equation*}
f_0\ix{x_0} \coloneqq \begin{cases}
f\ix{x_0 + s_\text{min}}, & x_0 \in \inds{S}_0, \\
0 & \quad \text{else},
\end{cases}
\quad \quad \quad
\varphi_0\ix{z_0} \coloneqq \begin{cases}
\varphi\ix{z_0 + g_\text{min}}, & z_0 \in \inds{G}_0, \\
0 & \quad \text{else}
\end{cases}
\end{equation*}
and $\inds{G}\coloneqq\inds{T}-\inds{S}$, $\inds{S}_0 \coloneqq \inds{S} - s_\text{min}$, $\inds{T}' \coloneqq \inds{T} - t_\text{min} + s_\text{max}-s_\text{min}$, and $\inds{G}_0 \coloneqq \inds{G} - g_\text{min}$. Thus one can apply a block of a convolution operator to a vector in work that scales nearly linearly with the linear dimensions of the block: $O(\sigma \log \sigma)$ where $\sigma = |\inds{S}| + |\inds{T}|$. To apply $\widetilde{A}\ix{\inds{T},\inds{S}}$ or $\widetilde{A}^*\ix{\inds{T},\inds{S}}$ to a vector, we use this method for each convolution in the sums (\eqref{eq:conv_op_approx_main} and \eqref{eq:adjoint_random_z}) defining $\widetilde{A}$ or $\widetilde{A}^*$, respectively, that could be non-zero. Since the functions $w_k$ are supported on the sets $U_k$, the terms in these sums that could be non-zero correspond to sets $U_k$ that intersect $\inds{S}$ when multiplying $\widetilde{A}\ix{\inds{T}, \inds{S}}$ with a vector, and $\inds{T}$ when multiplying with $\widetilde{A}^*\ix{\inds{T}, \inds{S}}$ with a vector. As a result, it costs
\begin{equation}
\underbrace{O(r_\inds{S}~ \sigma \log \sigma)}_{f \mapsto \widetilde{A}\ix{\inds{T}, \inds{S}} f} \quad\quad \text{and} \quad\quad \underbrace{O(r_\inds{T}~ \sigma \log \sigma)}_{f \mapsto \widetilde{A}^*\ix{\inds{T}, \inds{S}} f}
\end{equation}
work to apply $\widetilde{A}\ix{\inds{T}, \inds{S}}$ and $\widetilde{A}^*\ix{\inds{T}, \inds{S}}$ to vectors, respectively. Here $r_\inds{S}$ and $r_\inds{T}$ are the number of sets $U_k$ that intersect $\inds{S}$ and $\inds{T}$, respectively.

\subsection{Conversion to hierarchical matrix format}
\label{sec:H_matrix_actual}
Construction of a hierarchical matrix proceeds in the following steps:
\begin{enumerate}
\item The degrees of freedom are partitioned hierarchically into a cluster tree. 
\item The matrix entries are partitioned hierarchically into a block cluster tree.
\item A low-rank approximation is constructed for each block of the matrix that is marked as low-rank (i.e., admissible) within the block cluster tree.
\end{enumerate}
The $H$-matrix construction process is scalable if we can construct low-rank approximations (see Section \ref{sec:low_rank}) of the low-rank blocks (Step 3) in work that scales nearly linearly with the dimensions of the block. The method for efficiently applying blocks of $\widetilde{A}$ and $\widetilde{A}^*$ to vectors, outlined in Section \ref{sec:subblock_conv_fast}, allows us to do this using Krylov methods or randomized SVD. Whenever the Krylov method or randomized SVD requires the application of a block or its adjoint to a vector, we perform this computation using the method in Section \ref{sec:subblock_conv_fast}. Alternatively, formula \eqref{eq:conv_matrix_entries} for the matrix entries of $\widetilde{A}$ allows us to construct a low-rank approximation of a block by forming a CUR approximation, as is done in \cite{Bebendorf00,BormGrasedyck05,Tyrtyshnikov00}. Whenever the CUR approximation algorithm requires a row, column, or entry of the block, we access it using \eqref{eq:conv_matrix_entries}.

Since applying the block $\widetilde{A}\ix{\inds{T},\inds{S}}$ to a vector costs $O(\sigma \log \sigma)$ work, where $\sigma = |\inds{S}| + |\inds{T}|$, whereas accessing a row or column costs $O(\sigma)$ work, the CUR approach is asymptotically more scalable than the Krylov or randomized SVD approaches by a log factor. However, the CUR approach is less robust, and typically has poorer dependence on the rank of the blocks. In either case the overall cost of constructing the $H$-matrix scales nearly linearly with $N$. Moreover, the construction process only uses the approximation, $\widetilde{A}$. It does not require expensive application of $A$.

\section{Theory}
\label{sec:error_analysis}

Here we show that the error in $\widetilde{A}$ is controlled by the failure of $A$ to be locally translation-invariant with respect to a locally expanded cover, $\{U_k^E\}_{k=1}^r$, created by unioning each $U_k$ with its neighbors:
\begin{equation*}
U_k^E \coloneqq \bigcup_{j \in\nbrs(k)} U_j.
\end{equation*}
This provides an a-priori error estimate for the approximation, and shows that the approximation will not contain boundary artifacts. 

Let $F_k$ be the following functions that measure how much the impulse response of $A$ at $p_k$ fails to represent the impulse response of $A$ at $x$ (see Figure \ref{fig:translation_invariance}):
\begin{equation}
\label{eq:translation_failure_operator}
F_k\ix{y,x}\coloneqq A\ix{y-x+p_k,p_k} - A\ix{y, x}.
\end{equation}
We aggregate these $F_k$ to form a function $F$ which measures, pointwise, how much $A$ fails to be locally translation-invariant with respect to the cover $\{U_k^E\}_{k=1}^r$. Specifically, we define
\begin{equation}
\label{eq:defn_of_F}
F\ix{y,x} := \max_{k:(y,x) \in \mu_k^E} ~|F_k\ix{y,x}|,
\end{equation}
where the sets
\begin{equation}
\label{eq:skew_cylinder_set}
\mu_k^E\coloneqq\{(y,x): x \in U_k^E, y \in \inds{\Omega}, y - x + p_k \in \inds{\Omega}\}
\end{equation}
are defined to be all $(y,x) \subset \inds{\Omega} \times \inds{\Omega}$ such that $x \in U_k^E$, and $F_k\ix{y,x}$ is well-defined without resorting to extension by zero. In Theorem \ref{thm:a_priori_convolution_error} we will show that 
\begin{equation}
\label{eq:thm5_restatement}
\|\widetilde{A} - A\| \le \|F\|.
\end{equation} 
If we instead maximized over $k:x \in U_k^E$ rather than $k: (y,x) \in \mu_k^E$ in \eqref{eq:defn_of_F}, then the right hand side of bound \eqref{eq:thm5_restatement} would be undefined, because evaluating $\|F\|$ requires evaluating $A\ix{y-x+p_k,p_k}$, and $y-x+p_k$ may be outside of $\inds{\Omega}$ even if $x$, $y$, and $p_k$ are in $\inds{\Omega}$. Extending $A$ by zero would make $\|F\|$ well-defined, and would make the theory simple, but then the bound would be unnecessarily large due to boundary artifacts. Achieving bound \eqref{eq:thm5_restatement} while maximizing over $k: (y,x) \in \mu_k^E$ in \eqref{eq:defn_of_F} requires the boundary extension procedure of Section \ref{sec:boundary_extension}, and is the reason why proving bound \eqref{eq:thm5_restatement} will require several pages rather than a few lines.

A multi-step path leads to Theorem \ref{thm:a_priori_convolution_error}. In Proposition \ref{prop:weighting_product_space} we show that $\widetilde{A}$ can be reinterpreted as a weighted sum involving the original (not extended) impulse response functions $\varphi_k$, but with weighting functions that form a partition of unity on $\inds{\Omega}\times\inds{\Omega}$, and are supported in the sets $\mu_k^E$. Proposition \ref{prop:weighting_product_space} relies on a lemma about the functions $v_k^{(j)}$ used in our impulse response extension procedure (Lemma \ref{lem:boundary_extension_weighting}), which in turn relies on a lemma about Minkowski sums of boxes (Lemma \ref{lem:box_lemma}). After establishing these prerequisites, in Proposition \ref{prop:AAWF} we show that $\widetilde{A}-A$ can be represented as a weighted sum of the $F_k$ functions, with the same weighting functions as in Proposition \ref{prop:weighting_product_space}. Finally, we use Proposition \ref{prop:AAWF} and the properties of these weighting functions to prove bound \eqref{eq:thm5_restatement} in Theorem \ref{thm:a_priori_convolution_error}.

\begin{lem}
\label{lem:box_lemma}
If $\inds{S}$ and $\inds{T}$ are boxes, and $\inds{S}$ is at least as large as $\inds{T}$ in the sense that $s_\text{max}^i - s_\text{min}^i \ge t_\text{max}^i - t_\text{min}^i$ for $i=1,\dots,d$, then $\inds{S}+\inds{T} = \inds{S} + \corners(\inds{T})$.
\end{lem}

\begin{lem}
\label{lem:boundary_extension_weighting}
We have
\begin{equation}
\label{eq:vj_support_lem}
\sum_{j\in \nbrs(k)} v_{k}^{(j)}\ix{z}=\begin{cases}
1, & z \in \inds{\Omega} - U_k, \\
0, & \text{otherwise}.
\end{cases}
\end{equation}

\end{lem}
\begin{proof}
By construction,
\begin{equation*}
\sum_{j\in \nbrs(k)} v_{k}^{(j)}\ix{z}=\begin{cases}
1, & z \in \supp(c_k), \\
0, & \text{otherwise},
\end{cases}
\end{equation*}
and $\supp(c_k) = \bigcup\limits_{j \in \nbrs(k)}\left(\inds{\Omega}-p_j\right)$. We now show that 
$\inds{\Omega} - U_k = \bigcup\limits_{j \in \nbrs(k)}\left(\inds{\Omega}-p_j\right).$
To that end, recall that $U_k$ is the union of leaf boxes $\inds{C}_i$ that contain $p_k$. Thus
\begin{equation*}
\inds{\Omega} - U_k = \inds{\Omega} - \bigcup_{\inds{C}_i \in \cells(p_k)} \inds{C}_i = \bigcup_{\inds{C}_i \in \cells(p_k)} \left(\inds{\Omega} - \inds{C}_i\right).
\end{equation*}
Since $\inds{C}_i \subset \inds{\Omega}$, we see that $\inds{\Omega}$ is at least as large as $-\inds{C}_i$ (in the sense of Lemma \ref{lem:box_lemma}). Applying Lemma \ref{lem:box_lemma} to $\inds{\Omega}-\inds{C}_i$ and performing algebraic manipulations yields:
\begin{equation*}
\bigcup_{\inds{C}_i \in \cells(p_k)} \left(\inds{\Omega} - \inds{C}_i\right) = \bigcup_{\inds{C}_i \in \cells(p_k)} \left(\inds{\Omega} - \corners\left(\inds{C}_i\right)\right)
= \inds{\Omega} - \bigcup_{\inds{C}_i \in \cells(p_k)} \corners\left(\inds{C}_i\right).
\end{equation*}
Furthermore, by definition the union of all corners of leaf cells containing a point is the union of all neighboring points, so we have
\begin{equation*}
\inds{\Omega} - \bigcup_{\inds{C}_i \in \cells(p_k)} \corners\left(\inds{C}_i\right) 
= \inds{\Omega} - \bigcup_{j \in \nbrs(k)} p_j 
= \bigcup_{j \in \nbrs(k)} \left(\inds{\Omega} - p_j\right),
\end{equation*}
which, with the chain of set equalities in previous lines, implies the desired result.
\end{proof}

\begin{prop}
\label{prop:weighting_product_space}
Let
\begin{equation}
\label{eq:Wj_def}
W_k\ix{y,x} \coloneqq \sum_{j\in\nbrs(k)} w_j\ix{x} v_{j}^{(k)}\ix{y-x}.
\end{equation}
\begin{enumerate}
\item The entries of $\widetilde{A}$ can be written as:
\begin{equation*}
\widetilde{A}\ix{y,x} = \sum_{k=1}^r W_k\ix{y,x} \varphi_k\ix{y-x}.
\end{equation*}
\item The functions $\{W_k\}_{k=1}^r$ form a partition of unity:
\begin{equation*}
\sum_{k=1}^r W_k\ix{y,x}=1 \quad\text{for all } (y,x)\in \inds{\Omega}\times \inds{\Omega}
\end{equation*}
\item The partition of unity is subordinate to the cover $\{\mu_k^E\}_{k=1}^r$:
\begin{equation*}
\supp(W_k)\subset\mu_k^E.
\end{equation*}
\end{enumerate}
\end{prop}

\begin{proof}~
\paragraph{1} Substituting the definition of $\varphi_k^E$ from \eqref{eq:def_varphiE} into the definition of $\widetilde{A}$ from \eqref{eq:conv_op_approx_main} then performing algebraic manipulations, we have:
\begin{align}
\widetilde{A}\ix{y,x} &= \sum_{k=1}^r w_k\ix{x} \sum_{j \in \nbrs(k)} v_{k}^{(j)}\ix{y-x} \varphi_j\ix{y-x} \nonumber \\
&= \sum_{k=1}^r \sum_{j \in \nbrs(k)} w_k\ix{x} v_{k}^{(j)}\ix{y-x} \varphi_j\ix{y-x} \nonumber \\
&= \sum_{j=1}^r \sum_{k \in \nbrs(j)} w_k\ix{x} v_{k}^{(j)}\ix{y-x} \varphi_j\ix{y-x}
= \sum_{j=1}^r W_j\ix{y,x} \varphi_j\ix{y-x}. \nonumber
\end{align}
Going from the second to the third line we used the fact that
\begin{equation*}
\sum_{a\in X} \sum_{\{b : b\in X, b \sim a\}} f(a,b) = \sum_{b\in X} \sum_{\{a : a \in X, a \sim b\}} f(a,b)
\end{equation*}
for any symmetric relation $\sim$. Note the switch of $k$ and $j$.

\paragraph{2} Using the definition of $W_k$ in \eqref{eq:Wj_def}, we have
\begin{align*}
\sum_{k=1}^r W_k\ix{y,x} &= \sum_{k=1}^r \sum_{j\in\nbrs(k)} w_j\ix{x} v_{j}^{(k)}\ix{y-x} \\
&= \sum_{j=1}^r \sum_{k\in\nbrs(j)} w_j\ix{x} v_{j}^{(k)}\ix{y-x} 
= \sum_{j=1}^r w_j\ix{x} \left(\sum_{k\in\nbrs(j)} v_{j}^{(k)}\ix{y-x}\right).
\end{align*}
If $x \in U_j$ and $y\in \inds{\Omega}$, then Minkowski set arithmetic implies $y-x \in \inds{\Omega}-U_j$, so \eqref{eq:vj_support_lem} in Lemma \ref{lem:boundary_extension_weighting} implies
\begin{equation*}
\sum_{k\in\nbrs(j)} v_{j}^{(k)}\ix{y-x}=1.
\end{equation*}
Since $\supp(w_j) \subset U_j$, this implies
\begin{equation*}
\sum_{j=1}^r w_j\ix{x} \left(\sum_{k\in\nbrs(j)} v_{j}^{(k)}\ix{y-x}\right) = \sum_{j=1}^r w_j\ix{x} = 1.
\end{equation*}
Thus $\sum_{k=1}^r W_k\ix{y,x}=1$ as required.

\paragraph{3} From the definition of $v_k^{(j)}$ in \eqref{eq:vj_def}, either $\supp(v_{k}^{(j)}) = \left(\inds{\Omega} - p_j\right) \setminus \left(\inds{\Omega}-p_k\right)$ when $k \neq j$, or $\supp(v_{k}^{(j)})=\inds{\Omega}-p_j$ when $k=j$. In either case $\supp(v_{k}^{(j)}) \subset \inds{\Omega} - p_j$. Thus
\begin{equation*}
\left(y - x \notin \inds{\Omega} - p_j\right) \implies \left(v_{k}^{(j)}\ix{y - x}=0\right),
\end{equation*}
which is equivalent to the statement
\begin{equation}
\label{eq:vj0_is_implied}
\left(y - x + p_j \notin \inds{\Omega}\right) \implies \left(v_{k}^{(j)}\ix{y - x}=0\right).
\end{equation}
Since $W_k$ consists of a sum of terms, each term containing $v_{j}^{(k)}\ix{y - x}$, statement \eqref{eq:vj0_is_implied} implies (note the swap of $k,j$):
\begin{equation}
\label{eq:Wj_pyx_in_U}
\left(y - x + p_k \notin \inds{\Omega}\right) \implies \left(W_k\ix{y,x} =0\right).
\end{equation}
Additionally, since each $w_j$ in the sum defining $W_k$ is supported in the blocky neighborhood $U_j$, and since the union of these blocky neighborhoods $U_j$ is $U_k^E$, we have
\begin{equation}
\label{eq:Wj_x_in_UE}
\left(x \notin U_k^E\right) \implies \left(W_k\ix{y,x} =0\right).
\end{equation}
Altogether, \eqref{eq:Wj_pyx_in_U}, \eqref{eq:Wj_x_in_UE}, and the definition of $\mu_k^E$ in \eqref{eq:skew_cylinder_set} imply $\supp(W_k) \subset \mu_k^E$.
\end{proof}

\begin{prop}
\label{prop:AAWF}	
The pointwise error in our product-convolution approximation takes the following form:
\begin{equation}
\widetilde{A}\ix{y,x}-A\ix{y,x} = \sum_{k:(y,x) \in \mu_k^E} W_k\ix{y,x} F_k\ix{y,x}.
\end{equation}
\end{prop}

\begin{proof}
From Proposition \ref{prop:weighting_product_space} and the fact that $\varphi_k\ix{z}=A\ix{z+p_k,p_k}$, we know that 
\begin{equation*}
\widetilde{A}\ix{y,x} = \sum_{k=1}^r W_k\ix{y,x} A\ix{y-x+p_k,p_k},
\end{equation*}
Hence the pointwise error in the approximation takes the following form:
\begin{align*}
\widetilde{A}\ix{y,x}-A\ix{y,x} &= \sum_{k=1}^r W_k\ix{y,x} A\ix{y-x+p_k,p_k} - A\ix{y,x} \\
&= \sum_{k=1}^r W_k\ix{y,x} \left(A\ix{y-x+p_k,p_k} - A\ix{y,x}\right) \\
&= \sum_{k=1}^r W_k\ix{y,x} F_k\ix{y,x} = \sum_{k:(y,x) \in \mu_k^E} W_k\ix{y,x} F_k\ix{y,x}
\end{align*}
Going from the first line to the second line we used the partition of unity property of $W_k$ from Proposition \ref{prop:weighting_product_space}. Going from the second to the third line we used the definition of $F_k$. In the last equality on the third line we used the fact that $\supp\left(W_k\right) \subset \mu_k^E$.
\end{proof}

\begin{thm}
\label{thm:a_priori_convolution_error}
We have
\begin{equation}
\label{eq:thm5_overallbound}
\| \widetilde{A}-A\| \le \|F\|.
\end{equation}
\end{thm}

\begin{proof}
Using the result of Proposition \ref{prop:AAWF}, the fact that $W_k$ form a partition of unity, and the definition of $F$ yields the pointwise error bound
\begin{align*}
|\widetilde{A}\ix{y,x}-A\ix{y,x}| &= \left|\sum_{k:(y,x) \in \mu_k^E} W_k\ix{y,x} F_k\ix{y,x}\right| \\
&\le \max_{k:(y,x) \in \mu_k^E} |F_k\ix{y,x}| = F\ix{y,x}.
\end{align*}
The overall bound, \eqref{eq:thm5_overallbound}, follows directly from the definition of the norm and this pointwise bound.
\end{proof}

\begin{remark}
	Let 
	\begin{equation*}
	T\ix{y,x} := A\ix{y+x,x}
	\end{equation*}
	be the spatially varying impulse response function (see, e.g., \cite{BigotEscandeWeiss16} for a more in-depth discussion of the SVIR). Under the change of variables $h:=p-x$, $\xi:=y-x$, we may express the failure of local translation invariance in terms of the SVIR as follows:
	\begin{equation*}
	A\ix{y-x+p,p} - A\ix{y,x} = -\left(T\ix{\xi,p+h} - T\ix{\xi,p}\right).
	\end{equation*}
	If $x$ is near $p$, then $h$ is small, so
	\begin{equation*}
	T\ix{\xi,p+h} - T\ix{\xi,p} \approx \frac{dT}{dp}(\xi, p)h.
	\end{equation*}
	Hence, if our scheme is applied to a discretization of a continuous operator, the smoother the function $x \mapsto T\ix{y,x}$ is, the better our scheme will perform.
\end{remark}

\section{Numerical examples}
\label{sec:numerical_examples}

We numerically test our scheme on a spatially varying blur operator (Section \ref{sec:spatial_blur}), on the non-local component of the Schur complement associated with restricting the Poisson operator to an internal interface (Section \ref{sec:poisson_schur_mid}), and on the data misfit Hessian for an advection-diffusion inverse problem (Section \ref{sec:hessian}). For the spatially varying blur operator, our scheme refines towards the boundary between blur kernels and refines almost nowhere else, therefore outperforming the standard non-adaptive scheme which refines everywhere uniformly. For the Poisson interface Schur complement, our scheme is mesh scalable: it requires roughly the same convolution rank (number of terms in \eqref{eq:conv_op_approx_main}) to achieve a desired error tolerance regardless of how fine the mesh is. For the Hessian, our scheme is data scalable: it requires roughly the same convolution rank to achieve a desired error tolerance regardless of how informative the data are about the unknown parameter in the inverse problem. For both the Poisson Schur complement and the Hessian, we show that our scheme, in combination with $H$-matrix methods, can be used to build good preconditioners. 
Additionally, we find that the randomized a-posteriori error estimator achieves good performance with only a handful of random samples: our scheme performs almost as well with $q=5$ as it does with $q=100$.

For $H$-matrices, we use the standard coordinate splitting nested-bisection binary cluster tree\footnote{Degrees of freedom are split into two equally-sized clusters by a hyperplane normal to widest coordinate direction for that cluster. Then each cluster is split into two smaller clusters in the same way, and so on, recursively. The splitting continues until the number of degrees of freedom in a cluster is less than $32$.}, and the standard diameter-less-than-distance admissibilty condition\footnote{We mark a block of the matrix as low rank (admissible) if the distance between the degree of freedom cluster associated with the rows of the block and the diameter of the degree of freedom cluster associated with the columns of the block is less than or equal to the diameter of the smaller of the two degree of freedom clusters.}.

\subsection{Spatially varying blur}
\label{sec:spatial_blur}

\paragraph{Problem setup} Let $a$ be the following spatially varying blurring kernel,
\begin{equation*}
a(s,t) := \exp\left(-\frac{s^2 + t^2}{2 \sigma^2(s,t)}\right), \quad \text{where} \quad \sigma(s,t) = \begin{cases}
0.1, & s^2 + t^2 < 0.5, \\
0.2, & s^2 + t^2 \ge 0.5.
\end{cases}
\end{equation*}
Here $A$ is the matrix generated by sampling $a$ on $[-1,1]^2$ with a $75 \times 75$ equally spaced regular grid. 

\begin{figure}
\begin{subfigure}[b]{0.55\textwidth}
	\center
    \includegraphics[scale=0.55]{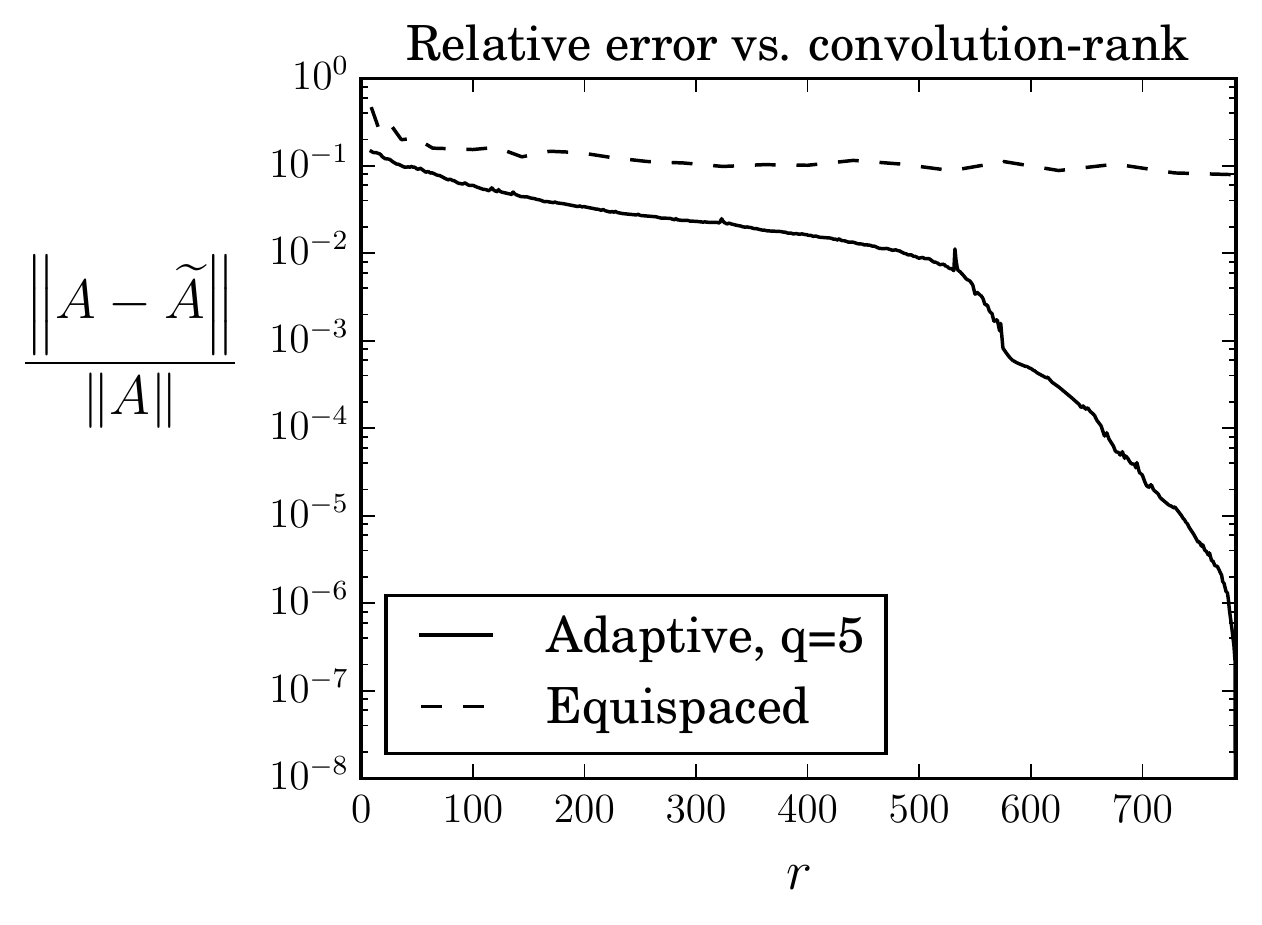}
    \caption{}
    \label{fig:spatially_varying_blur_convergence}
\end{subfigure}
\hfill
\begin{subfigure}[b]{0.44\textwidth}
	\center
    \includegraphics[scale=0.55]{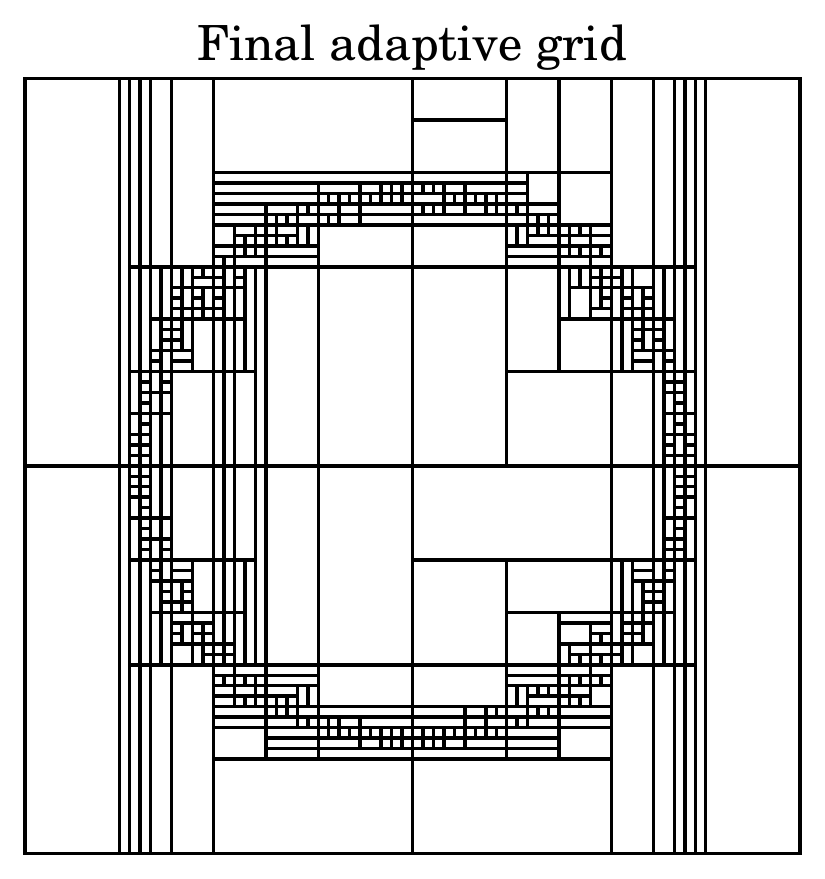}
    \caption{}
    \label{fig:spatially_varying_blur_grid}
\end{subfigure}
\caption{\textbf{Spatially varying blur:} Product-convolution approximation of the spatially varying blur operator defined in Section \ref{sec:spatial_blur}. (a) Convergence of our adaptive scheme, compared to convergence of standard product-convolution approximation with an equispaced regular grid of sample points and local bilinear interpolant weighting functions. (b) Final grid generated by our adaptive scheme.}
\label{fig:spatially_varying_blur}
\end{figure}

\paragraph{Results} Figure \ref{fig:spatially_varying_blur_convergence} compares product-convolution approximation of $A$ using our adaptive scheme, versus standard product-convolution approximation of $A$ using an equally spaced regular grid of sample points, with bilinear interpolation of impulse response functions, no adaptivity and no boundary extension procedure. Our adaptive scheme converges much faster than the regular grid scheme. 

Figure \ref{fig:spatially_varying_blur_grid} shows the final grid generated by our adaptive scheme, in which the boundary of the circle $s^2 + t^2 = 1$ is fully resolved with $2 \times 2$ cells. Error in the adaptive procedure is zero (within machine epsilon) for this final grid.

\subsection{Poisson interface Schur complement}
\label{sec:poisson_schur_mid}

\paragraph{Problem setup} Here we consider the discretized (negative) Laplace operator $K \approx -\Delta$ on the interior of the cube, $(-1,1)^3$. To build $K$, we discretize the Laplace operator on the whole cube, $[-1,1]^3$ with piecewise linear finite elements on a regular $n \times n \times n$ mesh of tetrahedra, so that there are $(n+1)^3$ mesh gridpoints. Then we exclude rows and columns from the resulting matrix that correspond to boundary degrees of freedom. The resulting $(n-1)^2 \times (n-1)^2$ matrix, $K$, is the coefficient matrix for the linear system that would need to be solved to determine the solution on the interior degrees of freedom for the Poisson problem in the cube with Dirichlet boundary conditions. 

Let `$i$' denote the degrees of freedom on the interface hyperplane at $z=0$ that separates\footnote{We choose $n$ even so that the interface is at $z=0$, rather than being slightly offset.} the degrees of freedom in the top half of the cube from the bottom half of the cube. Let `$t$' denote the degrees of freedom in the top half of the cube ($z > 0$), and let `$b$' denote degrees of freedom in the bottom half of the cube ($z < 0$), not including the interface in both cases. Denote the associated blocks of $K$ by $K_{it}$, $K_{tt}$, $K_{ti}$, $K_{ib}$ and so forth. We use our adaptive product-convolution scheme to approximate the operator
\begin{equation*}
A \coloneqq K_{it}K_{tt}^{-1}K_{ti} + K_{ib}K_{bb}^{-1}K_{bi}.
\end{equation*}
The matrix $-A$ is the non-local component of the Schur complement for degrees of freedom on the interface hyperplane, i.e., the matrix
\begin{equation*}
S \coloneqq K_{ii} - K_{it}K_{tt}^{-1}K_{ti} - K_{ib}K_{bb}^{-1}K_{bi}.
\end{equation*}
Matrix entries of $A$ are not directly available; we apply $A$ to vectors by performing matrix-vector products with $K_{bi}$, $K_{ib}$, $K_{ti}$, and $K_{it}$, and solving linear systems with $K_{tt}$ and $K_{bb}$ as the coefficient matrices. After approximating $A$ with $\widetilde{A}$ using our product-convolution scheme, we also construct the Schur complement approximation
\begin{equation*}
\widetilde{S} := K_{ii} - \widetilde{A}.
\end{equation*}
Such Schur complement approximations could be constructed recursively. One would subdivide the top and bottom subdomains, then subdivide the subdivisions, and so on. Approximations of Schur complements at deeper levels of the recursion would be used when constructing approximations at shallower levels. Here we only present results for one subdivision.

\begin{figure}
	\begin{minipage}[t]{0.54\textwidth}
		\mbox{}\\[-\baselineskip]
		\includegraphics[scale=0.48]{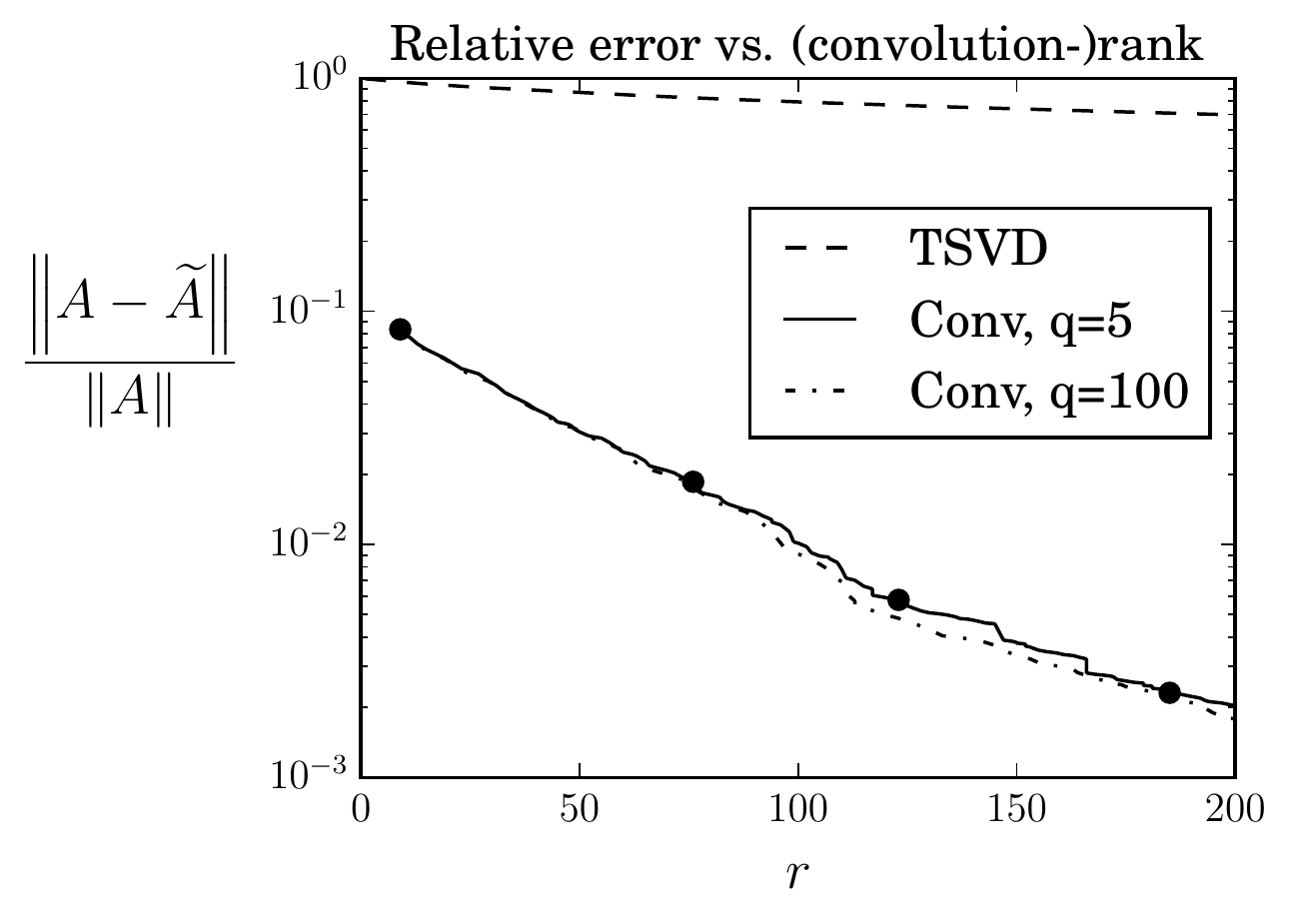}
	\end{minipage}\hfill
	\begin{minipage}[t]{0.43\textwidth}
		\mbox{}\\[-\baselineskip]
			\caption{\textbf{Poisson Schur complement:}  Relative error in truncated SVD low-rank approximation (`TSVD') compared to our product-convolution approximation ('Conv') as the (convolution) rank, $r$, changes. We show convergence curves for our scheme using both $q=5$ and $q=100$ random samples for the a-posteriori error estimator. Black dots correspond to the adaptive grids visualized in Figure \ref{fig:poisson_adaptive_grid}.}
	\label{fig:poisson_convergence_r}
	\end{minipage}
\end{figure}


\begin{figure}
	\begin{subfigure}[b]{0.24\textwidth}
		\center
		\includegraphics[scale=0.2]{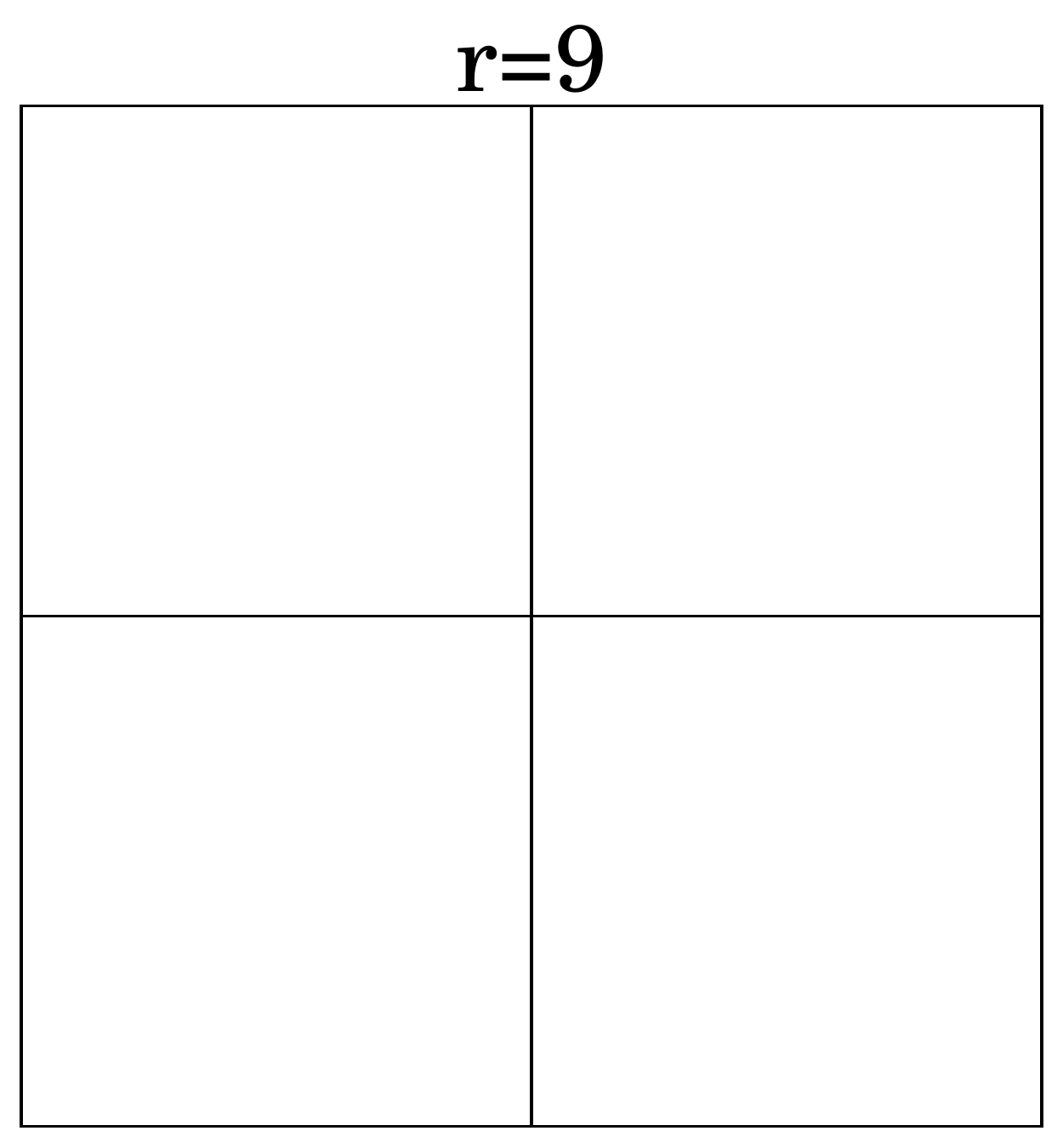}
		\label{fig:adaptive_grid_super_coarse}
	\end{subfigure}
	\begin{subfigure}[b]{0.24\textwidth}
		\center
		\includegraphics[scale=0.2]{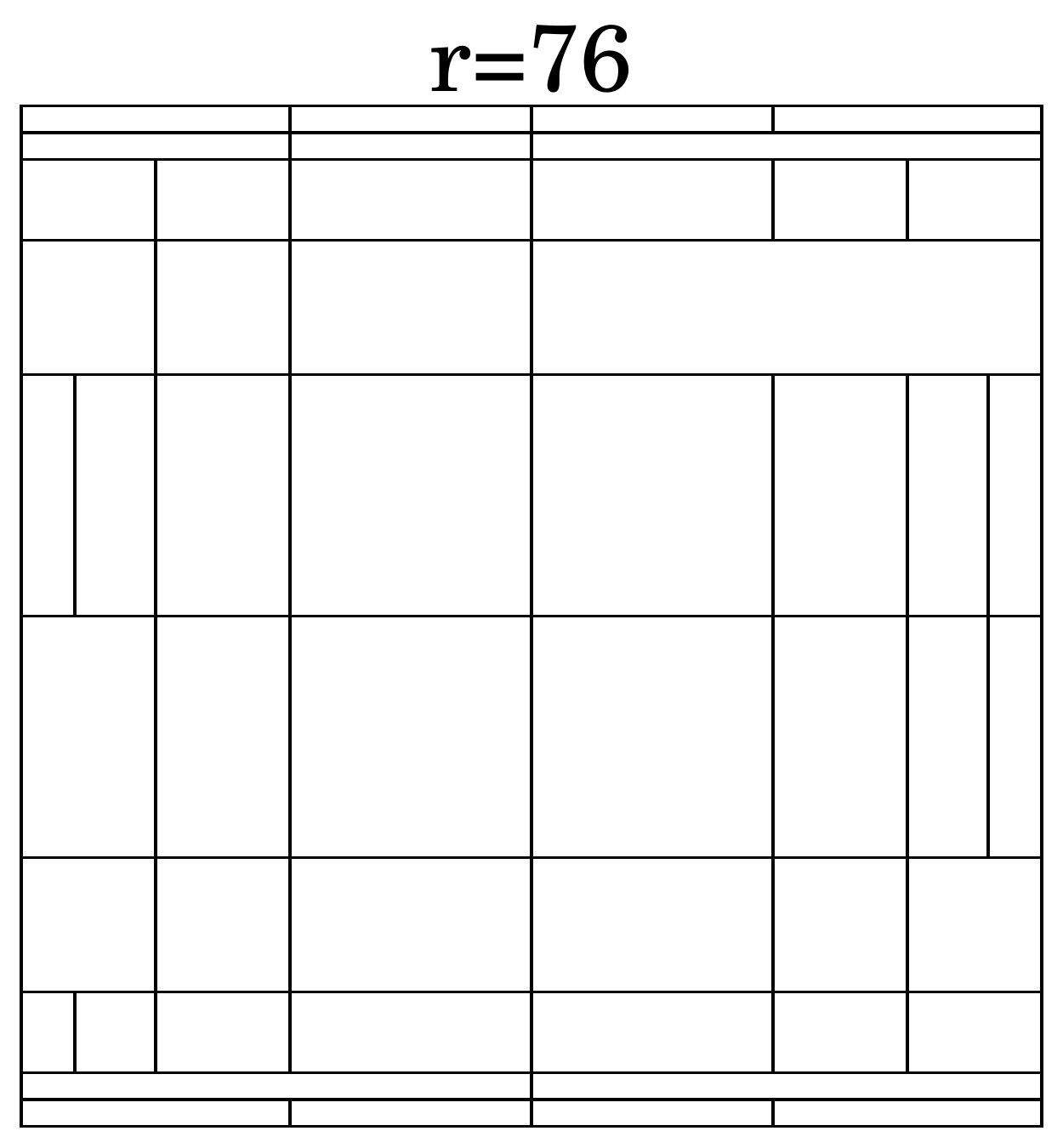}
		\label{fig:adaptive_grid_coarse}
	\end{subfigure}
	\begin{subfigure}[b]{0.24\textwidth}
		\center
		\includegraphics[scale=0.2]{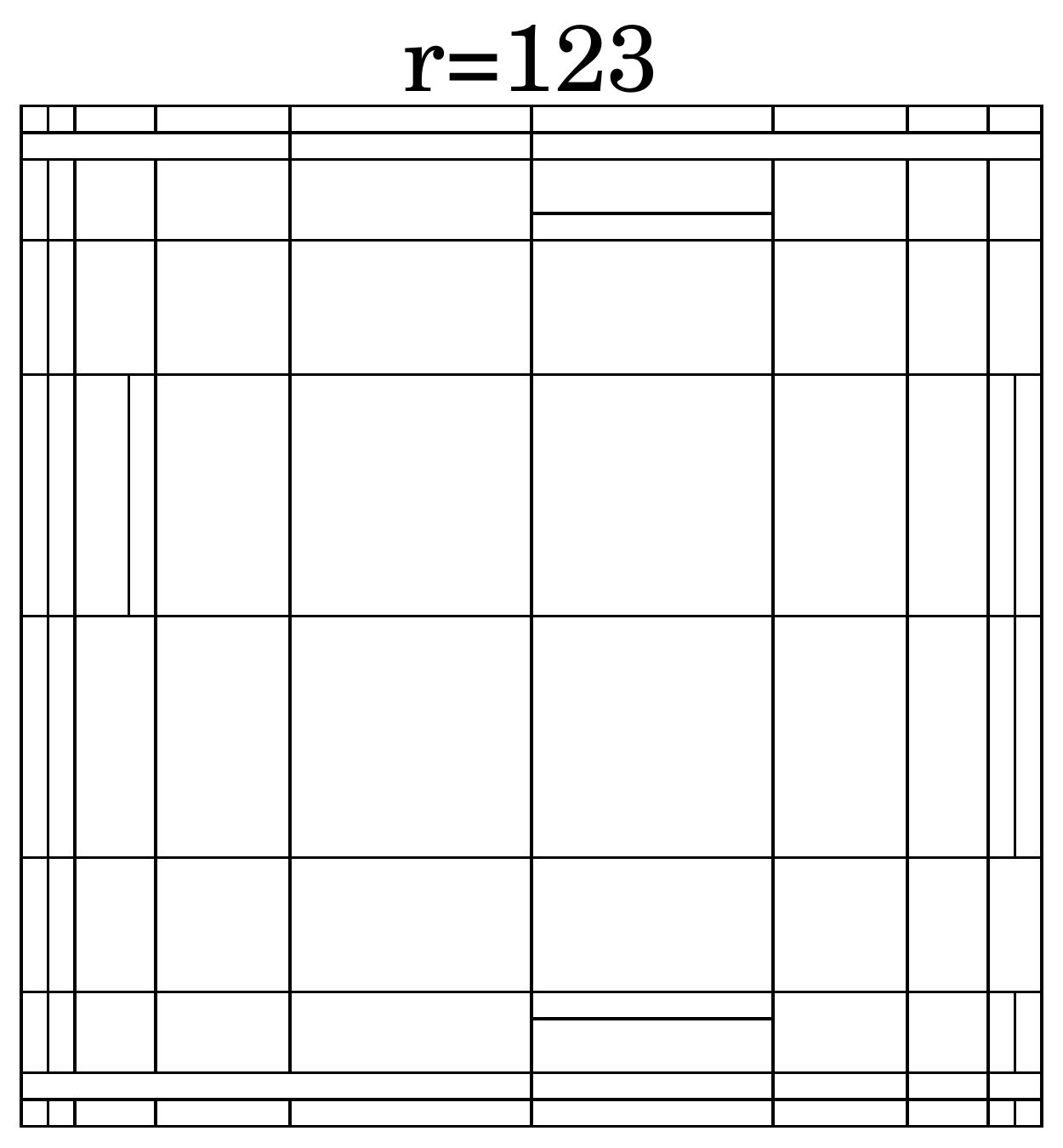}
		\label{fig:adaptive_grid_medium}
	\end{subfigure}
	\begin{subfigure}[b]{0.24\textwidth}
		\center
		\includegraphics[scale=0.2]{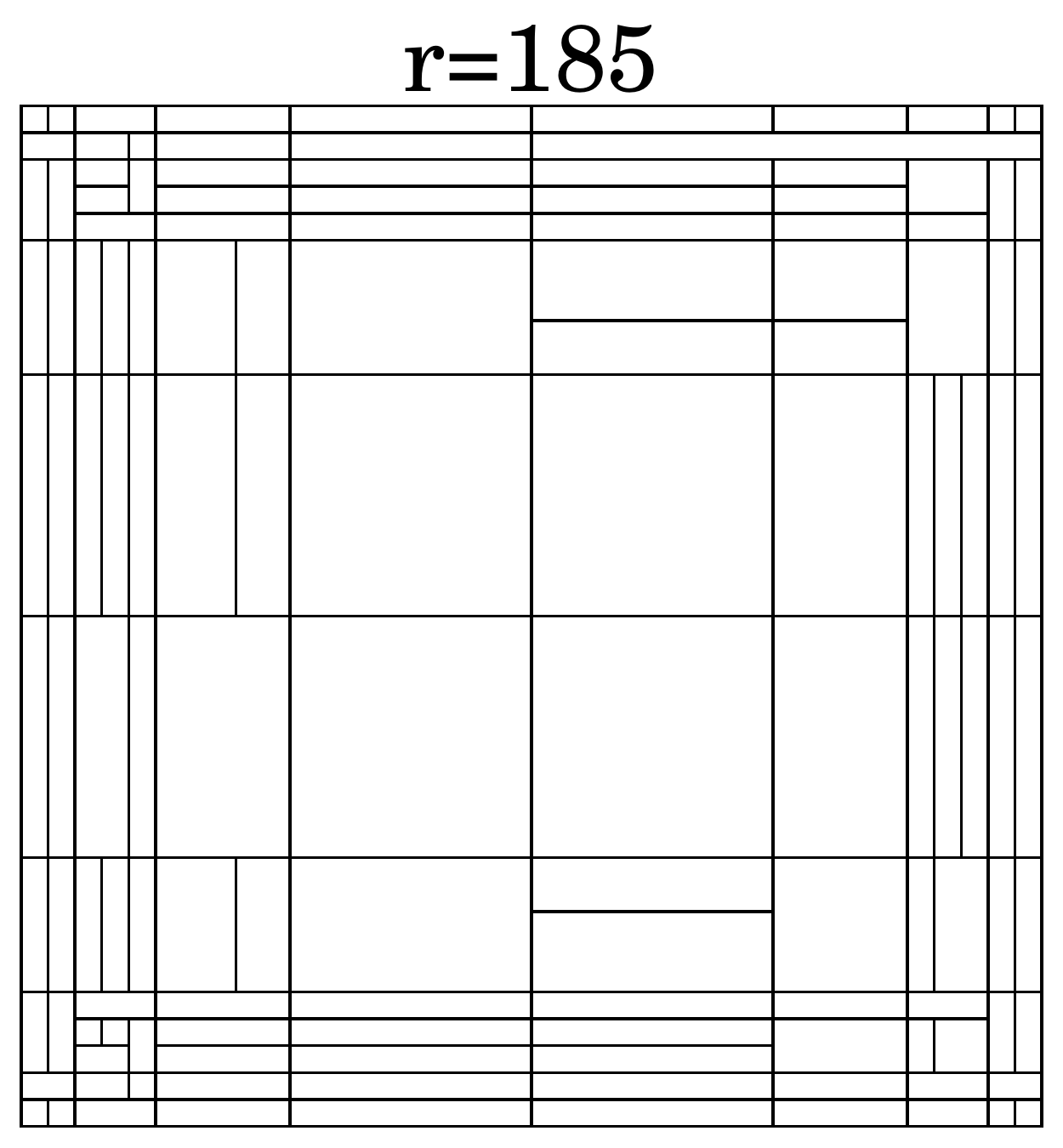}
		\label{fig:adaptive_grid_fine}
	\end{subfigure}
	\caption{\textbf{Poisson Schur complement:} Intermediate stages of adaptive grid refinement corresponding to black dots in Figure \ref{fig:poisson_convergence_r}.}
	\label{fig:poisson_adaptive_grid}
\end{figure}


\paragraph{Results} Figure \ref{fig:poisson_convergence_r} compares the convergence of our scheme to truncated SVD (`TSVD') approximation for $n=40$ ($N = (n-1)^2 = 1521$). Since the Poisson Schur complement is high rank, TSVD performs poorly. In contrast, our scheme performs well: at $r=200$ our scheme has less than $0.03$\% error, whereas TSVD has approximately $69$\% error. Figure \ref{fig:poisson_convergence_r} also shows that our scheme performs well even when we use a small number of random samples for the a-posteriori error estimator: the convergence curve for $q=5$ is almost identical to the convergence curve for $q=100$. Figure \ref{fig:poisson_adaptive_grid} displays the adaptive meshes from four different stages of the adaptive refinement process from Figure \ref{fig:poisson_convergence_r}. Our scheme adaptively refines towards the boundary, then the corners. This is expected since boundary effects are the only source of translation-invariance failure. 

Figure \ref{fig:poisson_mesh_scalability} compares our scheme to TSVD on a sequence of progressively finer meshes, from $h\approx 0.1$ to $h\approx 0.01$, where $h$ is the distance between adjacent gridpoints in the mesh. The curves show the (convolution) rank, $r$, required to achieve a relative error tolerance of $5$\%. The rank for TSVD grows with the number of degrees of freedom on the top surface ($r \sim O(1/h^2)$), offering little improvement over directly building a dense matrix representation of $A$ column-by-column. In contrast, the convolution rank for our scheme remains small for all $h$ considered.

\begin{figure}
	\begin{minipage}[t]{0.54\textwidth}
		\mbox{}\\[-\baselineskip]
		\includegraphics[scale=0.6]{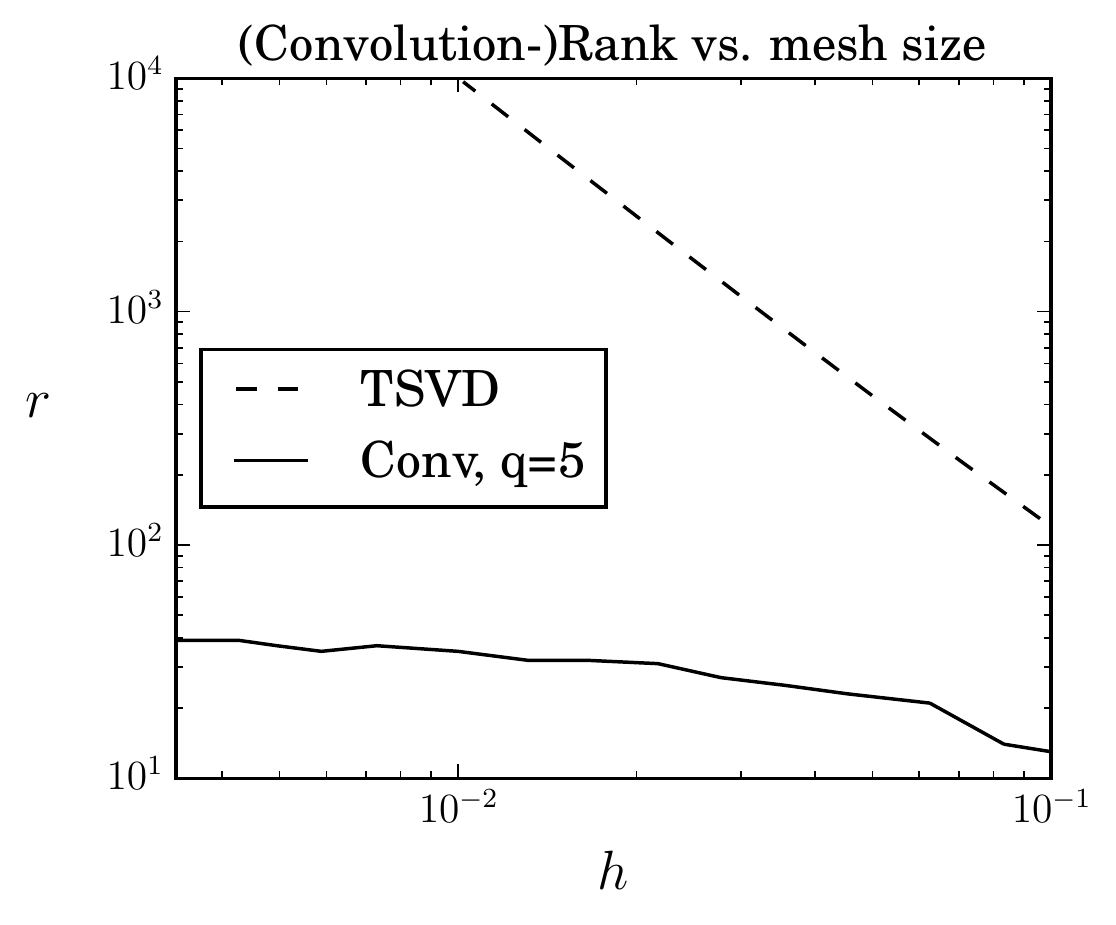}
	\end{minipage}\hfill
	\begin{minipage}[t]{0.43\textwidth}
		\mbox{}\\[-\baselineskip]
		\caption{\textbf{Poisson Schur complement:} The (convolution) rank, $r$, required to achieve a relative approximation error of $5$\%, for a variety of mesh sizes, $h$. `TSVD' indicates truncated SVD low rank approximation, and 'Conv' indicates our product-convolution scheme.} \label{fig:poisson_mesh_scalability}
	\end{minipage}
\end{figure}

\begin{figure}
	\begin{minipage}[t]{0.54\textwidth}
		\mbox{}\\[-\baselineskip]
		\includegraphics[scale=0.55]{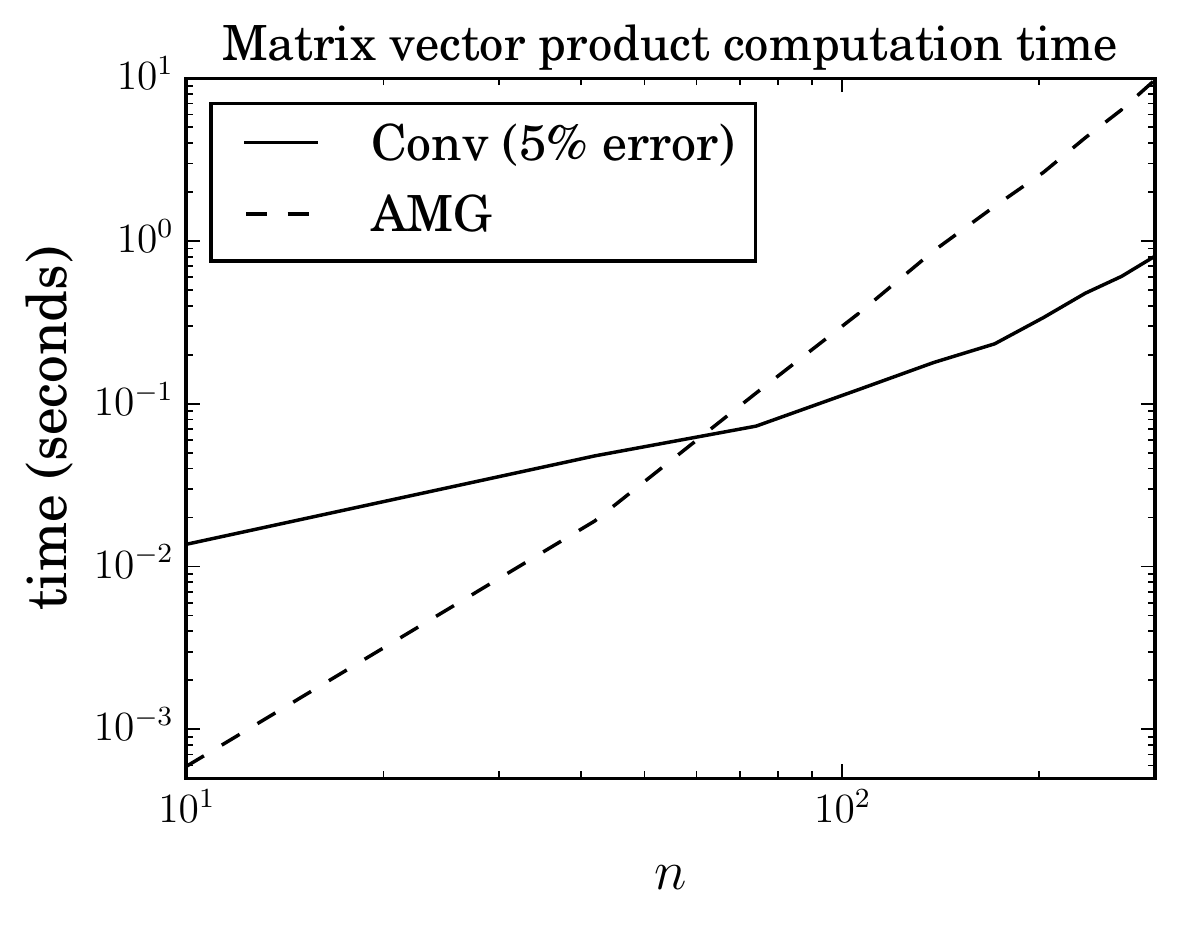}
	\end{minipage}\hfill
	\begin{minipage}[t]{0.43\textwidth}
		\mbox{}\\[-\baselineskip]
			\caption{\textbf{Poisson Schur complement:}  The time required to apply $\widetilde{A}$ to a vector using the FFT to compute the convolutions (`Conv'), compared to the time required to apply $A$ to a vector, using multigrid to apply the matrices $K_{tt}^{-1}$ and $K_{bb}^{-1}$ to vectors (`AMG'). For our product-convolution scheme, the average slope between $n=171$ and $n=300$ (from the final upturn to the end, containing 5 equally spaced $n$) is $2.2$, suggesting an asymptotic cost of $O(n^{2.2})$ (theory predicts $O(n^2 \log n)$). For algebraic multigrid, the average slope between $n=171$ and $n=300$ is $3.2$, suggesting an asymptotic cost of $O(n^{3.2})$ (theory predicts $O(n^3)$).} 
		\label{fig:poisson_matvec_timing}
		\end{minipage}		
\end{figure}

Figure \ref{fig:poisson_matvec_timing} compares the time required to apply $A$ to a vector, versus the time required to apply $\widetilde{A}$ to a vector. When applying $A$ to vectors, we solve the necessary linear systems with $K_{tt}$ and $K_{bb}$ as coefficient operators using PyAMG's \cite{OlSc2018} rootnode algebraic multigrid. When applying $\widetilde{A}$ to vectors, we use the FFT, as discussed in Section \ref{sec:apply_H_to_vectors}. For large $n$, applying $\widetilde{A}$ to a vector is much cheaper than applying $A$ to a vector.

In Table \ref{tbl:poisson_condition_number} we compare the condition number of the Schur complement, $S$, with the condition numbers of the preconditioned Schur complement, $\widetilde{S}^{-1} S$, for $n \times n \times n$ meshes ranging from $n=10$ to $n=100$. Here $\widetilde{S}^{-1}$ is constructed by converting $\widetilde{S}$ to $H$-matrix format, then inverting it using $H$-matrix arithmetic. Here, we use a tolerance of $10^{-6}$ for the low-rank approximations performed during $H$-matrix construction and arithmetic. The condition number of the (unpreconditioned) Schur complement grows as $O(1/h)$, where $h \approx 1/n$ is the mesh size. In contrast, the preconditioned Schur complement remains extremely well conditioned: the largest value of $\cond\left(\widetilde{S}^{-1} S\right)$ is $1.9$ for all meshes considered. 

\begin{table}
	\centering
	\begingroup
	\renewcommand*{\arraystretch}{1.0}
	\begin{tabular}{c|c|c|c}
		$n$ & $\cond \left(S\right)$ & $\cond\left(\widetilde{S}^{-1} S\right)$ & $r$\\
		\hline
		$10$  & $10.3$  & $1.1$ & $9$  \\
		$20$  & $21.3$  & $1.2$ & $20$ \\
		$30$  & $32.2$  & $1.3$ & $27$ \\
		$40$  & $43.0$  & $1.4$ & $28$ \\
		$50$  & $53.8$  & $1.5$ & $31$ \\
		$60$  & $64.5$  & $1.5$ & $33$ \\
		$70$  & $75.3$  & $1.8$ & $32$ \\
		$80$  & $86.1$  & $1.8$ & $35$ \\
		$90$  & $96.9$  & $1.8$ & $35$ \\
		$100$ & $107.7$ & $1.9$ & $35$
	\end{tabular}
	\endgroup
	\caption{\textbf{Poisson Schur complement:} Comparison of condition numbers for the Poisson interface Schur complement for a range of $n \times n \times n$ meshes. $S$ is the unpreconditioned Schur complement. $\widetilde{S}$ is the approximate Schur complement generated by replacing the nonlocal terms, $A$, within the Schur complement, with our convolution aproximation, $\widetilde{A}$, with a $5$\% relative error tolerance. The last column shows $r$, the convolution-rank of $\widetilde{A}$.}  
	\label{tbl:poisson_condition_number}
\end{table}

\subsection{Advection-diffusion inverse problem Hessian}
\label{sec:hessian}
\paragraph{Problem setup} In this section we approximate the data misfit portion of the Hessian for an advection-diffusion inverse problem in which an unknown initial concentration, $m$, of a contaminant, $u$, is inferred from time series data, $y$, of the contaminant flowing through a boundary, $\Gamma$. Specifically, consider the following PDE:
\begin{equation}
\label{eq:adv_diff_pde}
\begin{cases}
\frac{\partial u}{\partial t} =  \frac{1}{\mathrm{Pe}} \Delta u - \left(\begin{smallmatrix}0 \\ 1\end{smallmatrix}\right) \cdot \nabla u, & t \in [0,1], \\
u = m, & t=0,
\end{cases}
\end{equation}
where $\mathrm{Pe}$ is the Peclet number. The region of interest and support of $m$ is the unit square, $\Omega=[0,1]^2$, and the desired unbounded domain for the PDE is $\mathbb{R}^2$. To simulate the effect of having an unbounded domain, we extend the computational domain beyond $[0,1]^2$ on all sides and use Neumann boundary conditions on the outer, larger, domain. We use $y$ to denote the known noisy time series observations of $u$ on the top boundary: $y(x,t) = u(x,t) + \zeta, ~ x \in \Gamma, ~ t \in (0,1]$, where $\Gamma :=[0,1] \times \{1\}$ and $\zeta$ is $1$\% independent and identically distributed Gaussian noise. 

The inverse problem is: given $y$, determine $m$. This is commonly formulated as a least squares optimization problem of the following form:
\begin{equation}
\label{eq:optimization_problem}
\min_{m} ~ J(m) + R(m),
\end{equation}
where $J$ is the data misfit
\begin{equation*}
J(m) := \frac{1}{2}\int_0^1 \!\! \int_\Gamma \left(u(m) - y\right)^2 dx ~dt,
\end{equation*}
and $R(m)$ is a quadratic regularization term. Here $u(m)$ denotes the solution of \eqref{eq:adv_diff_pde} as a function of $m$. We use Laplacian regularization, $R(m) = \frac{\alpha}{2} m^T \Delta m$, where $\Delta$ is a discretization of the Laplacian operator with zero Dirichlet boundary conditions, and $\alpha=10^{-3}$ is the regularization parameter. This value of $\alpha$ was chosen since it satisfies the Morozov discrepancy principle \cite{morozov1984methods} to within a $5$\% tolerance for all Peclet numbers considered. For discretization, we use piecewise linear finite elements defined on a regular rectilinear $100 \times 100$ mesh of triangles, with $100$ time steps. We use backward Euler time stepping and SUPG stabilization \cite{BrooksHughes82}.

\begin{figure}
	\center
    \includegraphics[scale=0.5]{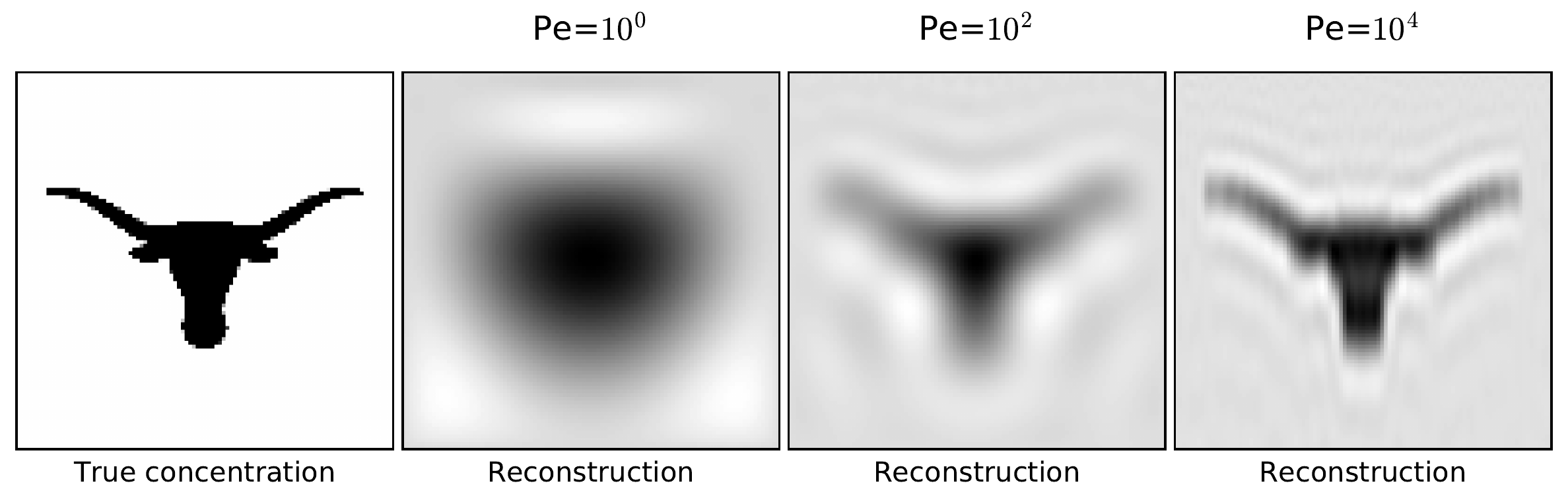}
    \caption{\textbf{Advection-diffusion inverse problem Hessian:} True initial concentration $m$ (left image) and optimal reconstructions for Peclet numbers $10^0$, $10^2$, and $10^4$ (all other images, left to right).  The regularization parameter used for the reconstruction is $\alpha=10^{-3}$, which satisfies the Morozov discrepancy principle within a $5$\% tolerance for all Peclet numbers considered.}
    \label{fig:advection_reconstructions}
\end{figure}

We use an image of the University of Texas ``Hook'em Horns'' logo as the initial concentration, $m$. The sharp edges in this image are computationally expensive to recover using existing methods. The solutions to the inverse problem for Peclet numbers in the range $10^0$ to $10^5$ are shown in Figure \ref{fig:advection_reconstructions}.

The results we present are twofold. First, we show that our product-convolution scheme can be used to approximate the discretized version of the operator
\begin{equation*}
A := \frac{d^2 J}{dm^2},
\end{equation*}
which is the Hessian of the data misfit. Second, we use the convolution approximation of $A$ to build a preconditioner for the overall Hessian,
\begin{equation*}
H := A + \frac{d^2 R}{dm^2}.
\end{equation*}
We show that the preconditioner is effective even if the Peclet number is large.

\paragraph{Preconditioning the Hessian} Since $y$ depends linearly on $m$, the solution to \eqref{eq:optimization_problem} is the solution to a linear system with $H$ as the coefficient matrix. Although this inverse problem is linear, Newton methods for solving nonlinear advection-diffusion inverse problems require solving linear systems with similar Hessians as coefficient operators. The Hessian of the regularization, $\frac{d^2 R}{dm^2}$, is a differential operator with known entries, and thus it is easy to manipulate. In contrast, $A$ is dense and its matrix entries are not directly available. We can only apply $A$ to vectors using an adjoint-based framework (see \cite{AkcelikBirosDraganescuEtAl05}). This requires solving a pair of advection-diffusion equations: a state equation of the form \eqref{eq:adv_diff_pde} forward in time, and the adjoint of \eqref{eq:adv_diff_pde} backward in time. Explicitly forming $A$ is thus prohibitively expensive: a pair of PDEs would need to be solved for every column of $A$. 

While Krylov methods can be used to solve linear systems with the Hessian as the coefficient operator in a matrix-free manner, good general purpose preconditioners have not been available (see \cite{AkcelikBirosGhattasEtAl06a} for a discussion of these issues). But now our convolution-product scheme allows us to build a good preconditioner as follows: first we form a product-convolution approximation of $A$, then convert it to $H$-matrix format, then symmetrize it, then add a small amount of identity regularization, then combine it with $\frac{d^2 R}{dm^2}$, then finally invert the combined $H$-matrix with fast $H$-matrix arithmetic. In detail, we form the following approximation to the inverse of the Hessian, which we use as a preconditioner:
\begin{equation}
\label{eq:adv_complete_preconditioner}
P^{-1} := \left((\widetilde{A} + \widetilde{A}^T)/2 + \tau \nor{A} I + \frac{d^2 R}{dm^2} \right)^{-1} \approx \left(A + \frac{d^2 R}{dm^2}\right)^{-1}.
\end{equation}
Here $\tau \nor{A}I$ is a small amount of additional regularization ($I$ is the identity matrix). We use $\tau=0.0025$. Matrix addition, scaling, and inversion in \eqref{eq:adv_complete_preconditioner} are performed with $H$-matrix arithmetic. Here, we use a fixed rank of $20$ for the low-rank approximations performed during $H$-matrix construction and arithmetic.

\paragraph{Data scalability} The Peclet number, $\mathrm{Pe}$, controls the ratio of advection to diffusion. As $\mathrm{Pe}$ increases, the rank of $A$ increases \cite{FlathEtAl11}, making the inverse problem more difficult to solve with existing methods. This increase in the rank corresponds to an increase in the informativeness of the data about the parameter in the inverse problem---eigenvectors of $A$ corresponding to large eigenvalues represent modes of the parameter that are well-informed by the data, whereas eigenvectors of $A$ corresponding to small eigenvalues represent modes of the parameter that are poorly-informed by the data (see \cite{AlgerEtAl17} for a discussion of these issues). As a result, for an approximation of $A$ to be data-scalable (perform well regardless of how informative the data are about the parameter), the cost of constructing the approximation must not grow as $\mathrm{Pe}$ increases.

\begin{figure}
	\begin{minipage}[t]{0.65\textwidth}
		\mbox{}\\[-\baselineskip]
	    \includegraphics[scale=0.60]{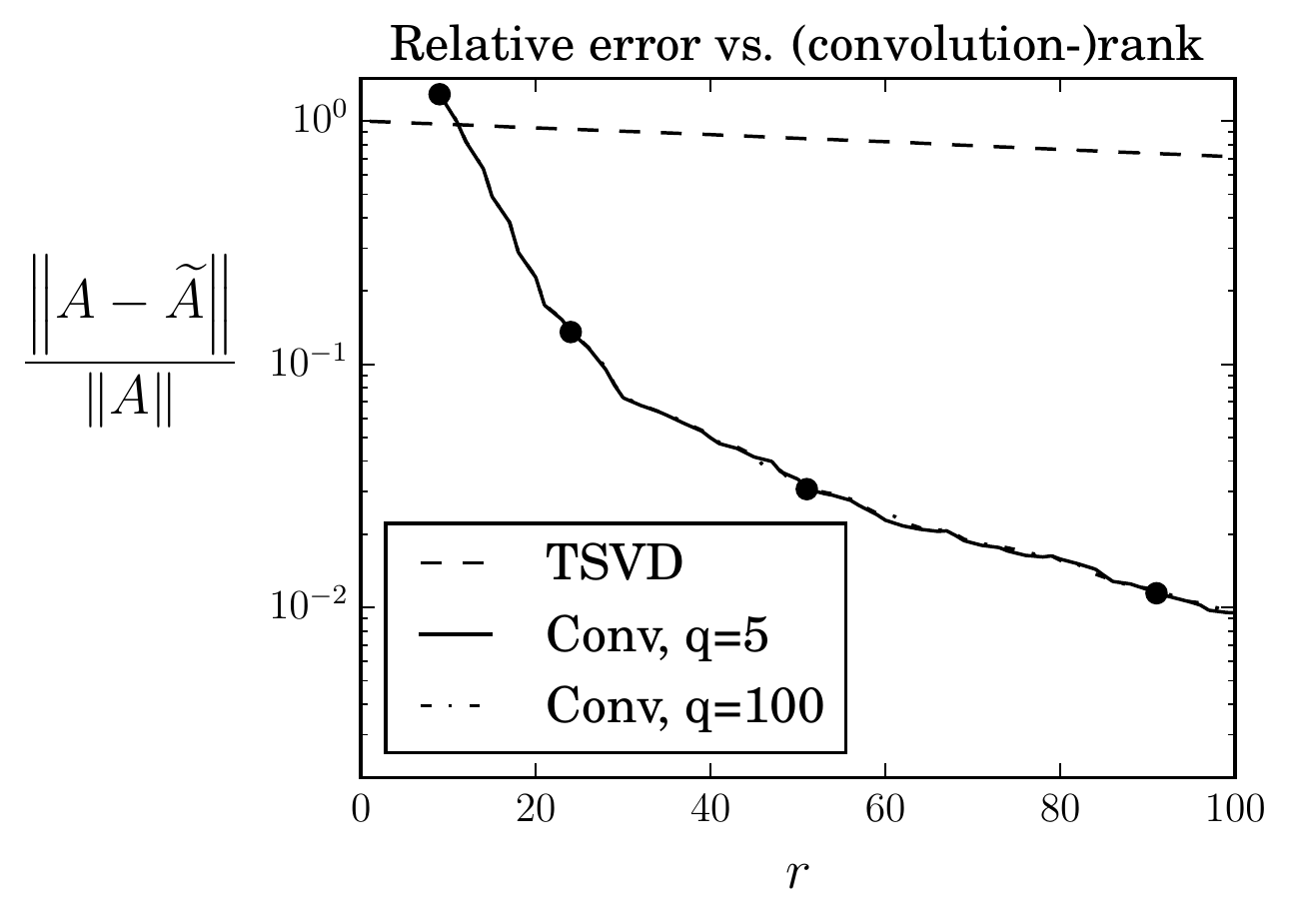}
	  \end{minipage}\hfill
	  \begin{minipage}[t]{0.35\textwidth}
		\mbox{}\\[-\baselineskip]
	    \caption{\textbf{Advection-diffusion inverse problem Hessian:} Relative error in the truncated SVD (`TSVD') low-rank approximation compared to our product-convolution approximation (`Conv') as the (convolution) rank, $r$, changes. We show convergence curves for our scheme using both $q=5$ and $q=100$ random samples for the a-posteriori error estimator. Black dots correspond to the adaptive grids visualized in Figure \ref{fig:adv_diff_adaptive_grid}.} \label{fig:adv_diff_convergence}
	  \end{minipage}
\end{figure}
\begin{figure}
\begin{subfigure}[b]{0.24\textwidth}
	\center
    \includegraphics[scale=0.2]{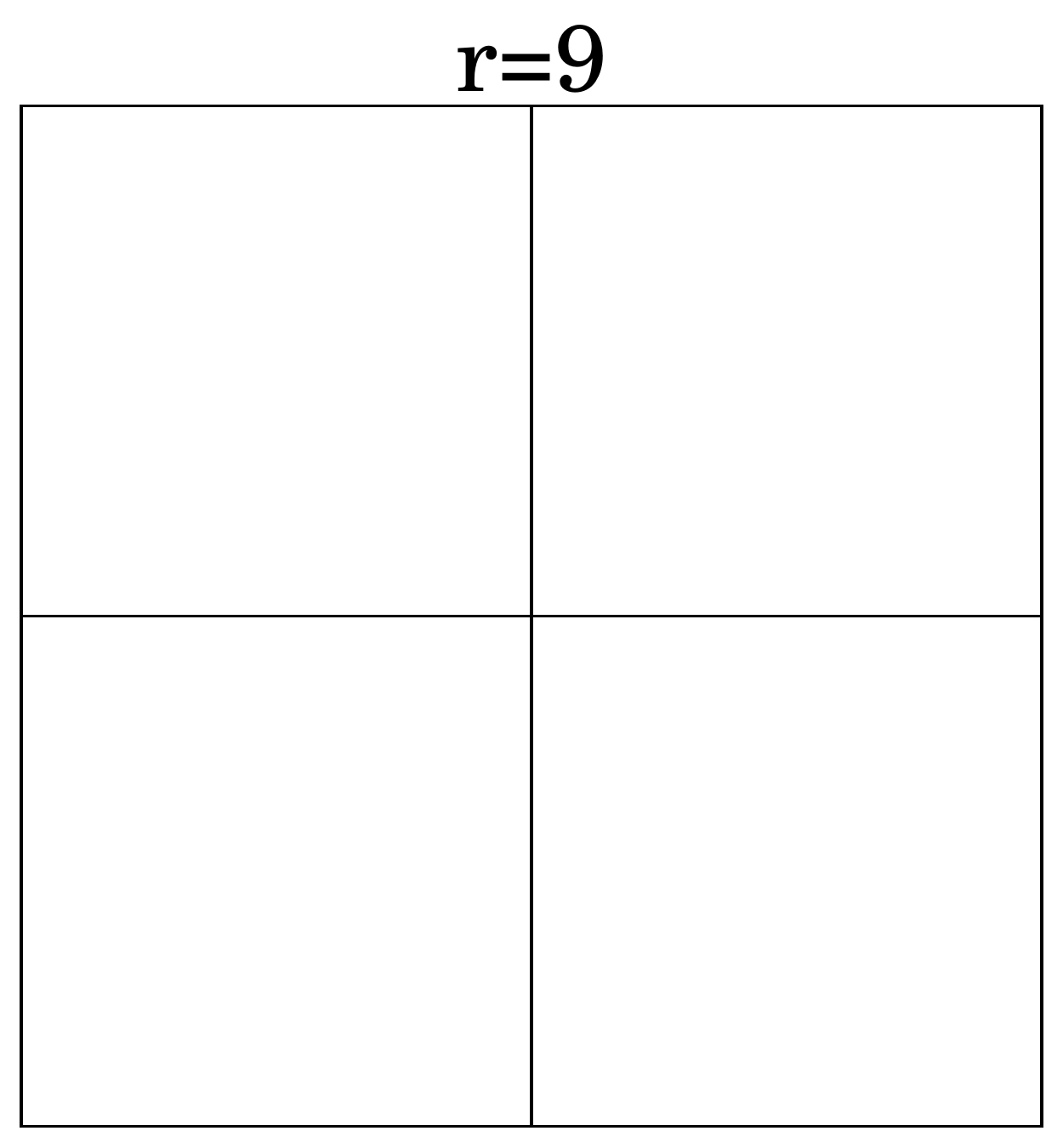}
    \label{fig:adv_diff_adaptive_grid_coarse}
\end{subfigure}
\begin{subfigure}[b]{0.24\textwidth}
	\center
    \includegraphics[scale=0.2]{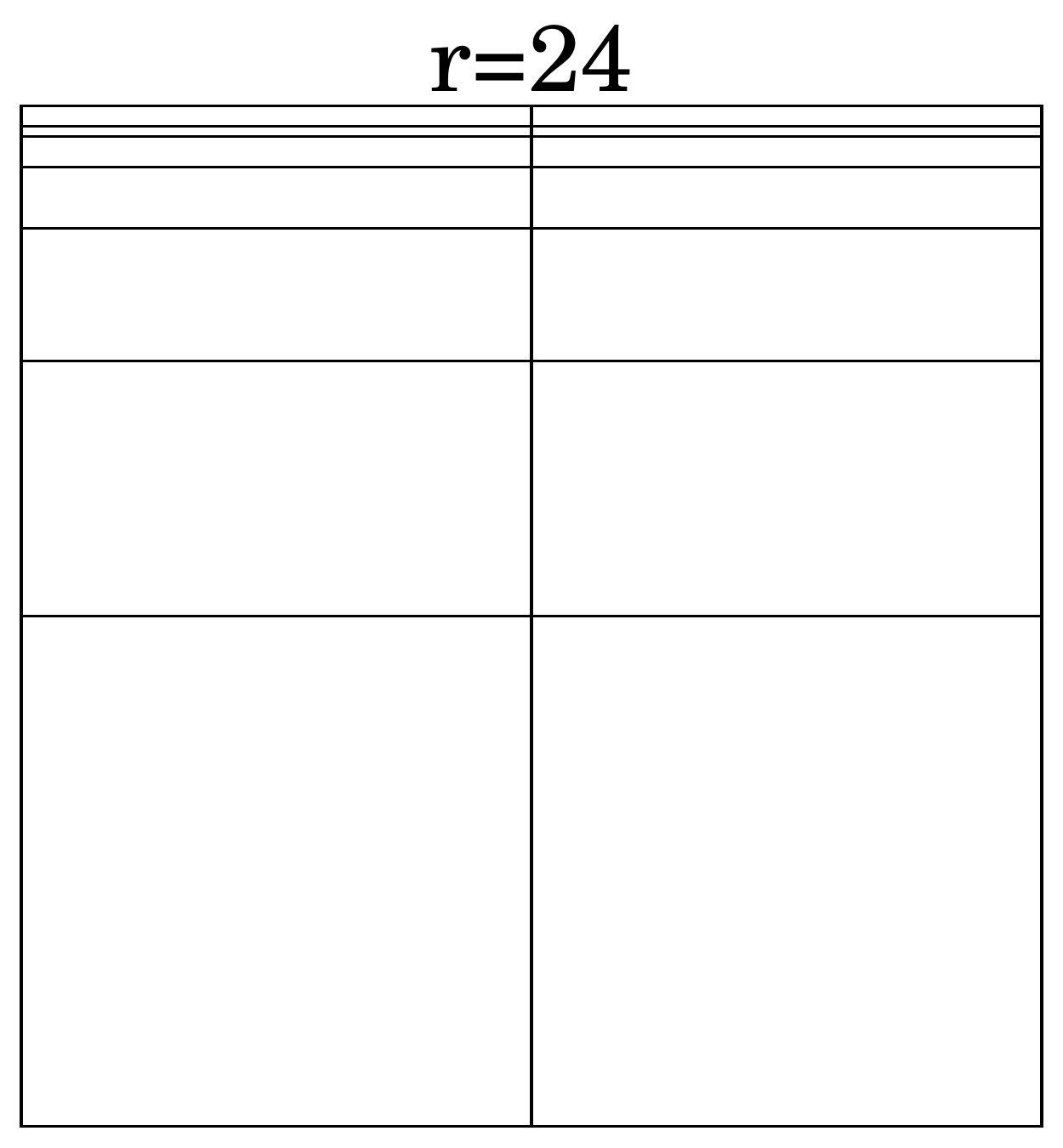}
    \label{fig:adv_diff_adaptive_grid_medium}
\end{subfigure}
\begin{subfigure}[b]{0.24\textwidth}
	\center
    \includegraphics[scale=0.2]{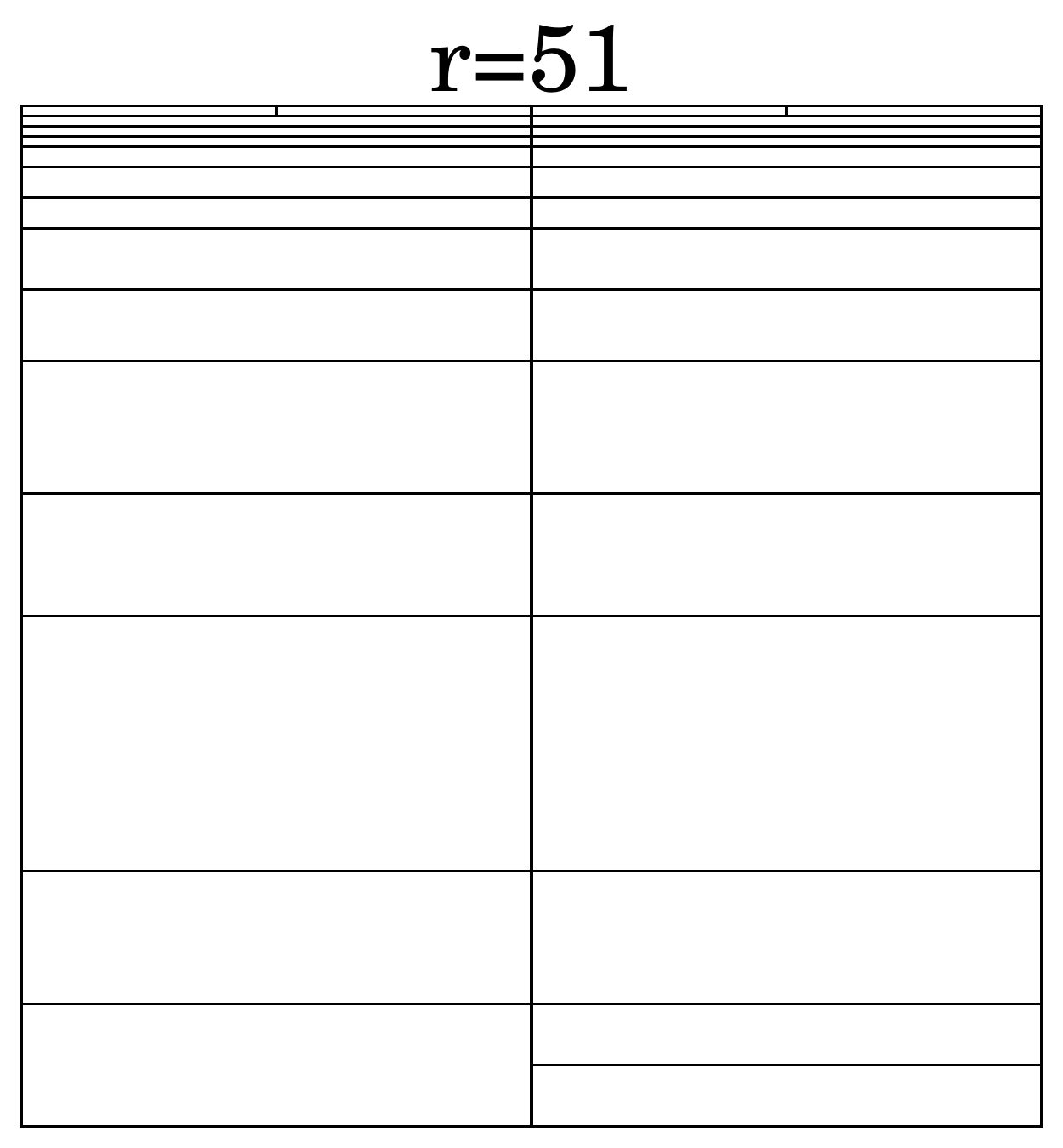}
    \label{fig:adv_diff_adaptive_grid_fine}
\end{subfigure}
\begin{subfigure}[b]{0.24\textwidth}
	\center
    \includegraphics[scale=0.2]{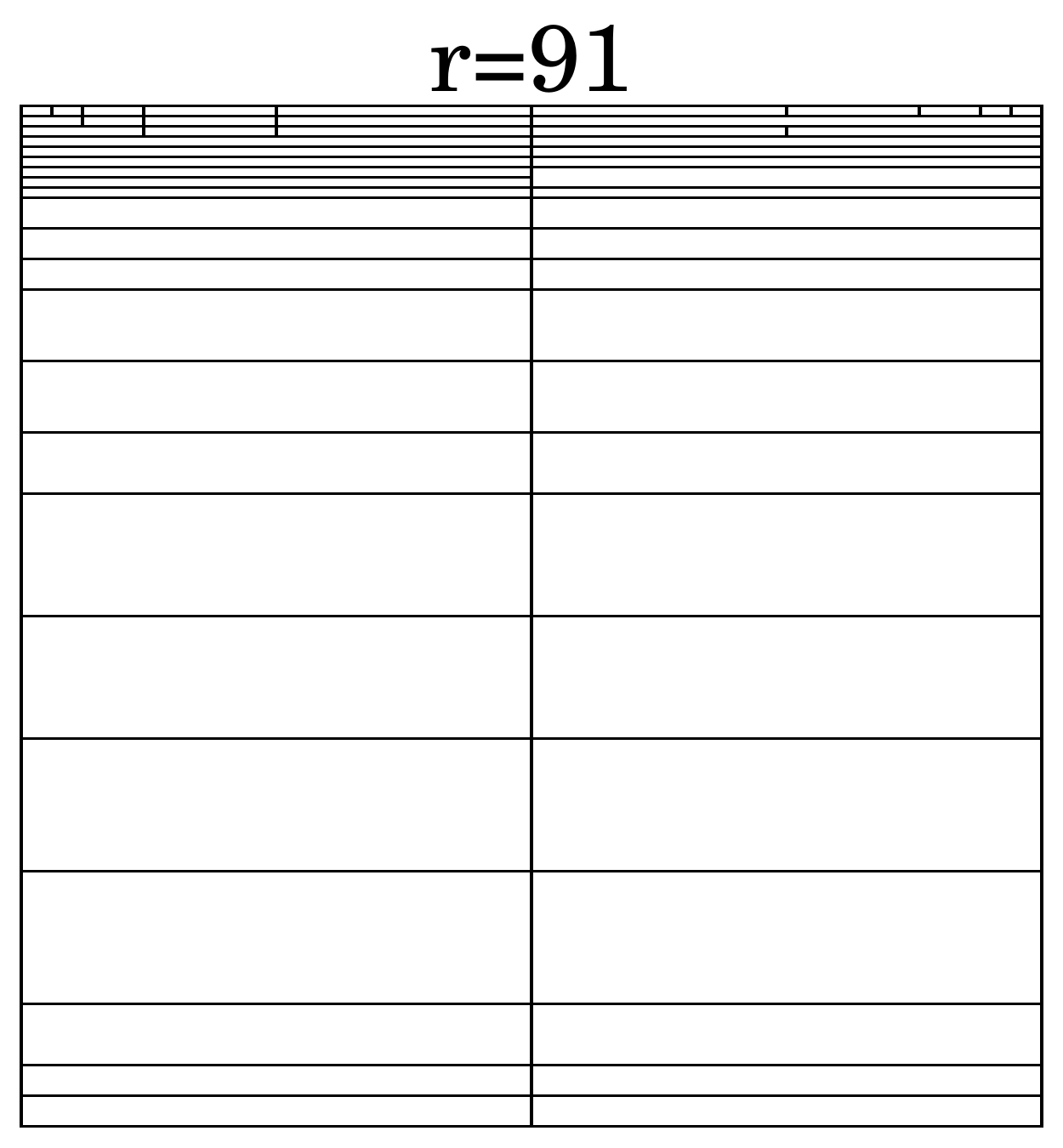}
    \label{fig:adv_diff_adaptive_grid_veryfine}
\end{subfigure}
\caption{\textbf{Advection-diffusion inverse problem Hessian:} Intermediate stages of adaptive grid refinement corresponding to black dots in Figure \ref{fig:adv_diff_convergence}.}
\label{fig:adv_diff_adaptive_grid}
\end{figure}

\paragraph{Results} Figure \ref{fig:adv_diff_convergence} compares the convergence of our product-convolution scheme (`CONV') to truncated SVD low rank approximation (`TSVD') when $\mathrm{Pe}=10^4$. Our scheme performs better than TSVD: at $r=100$ our scheme has less than $1$\% error whereas TSVD has approximately $71$\% error. Like the Poisson problem, the convergence curve for $q=5$ is almost identical to the convergence curve for $q=100$. Figure \ref{fig:adv_diff_adaptive_grid} shows the adaptive meshes from four different stages of the adaptive refinement process from Figure \ref{fig:adv_diff_convergence}. Our scheme chooses to adaptively refine in the direction of the vertical flow, prioritizing refinement near the top surface. We expect similar results would hold for inverse problems involving non-vertical, non-uniform flow if the convolution grid were aligned with the streamlines of the flow.

\begin{figure}
	\begin{minipage}[t]{0.54\textwidth}
		\mbox{}\\[-\baselineskip]
	    \includegraphics[scale=0.6]{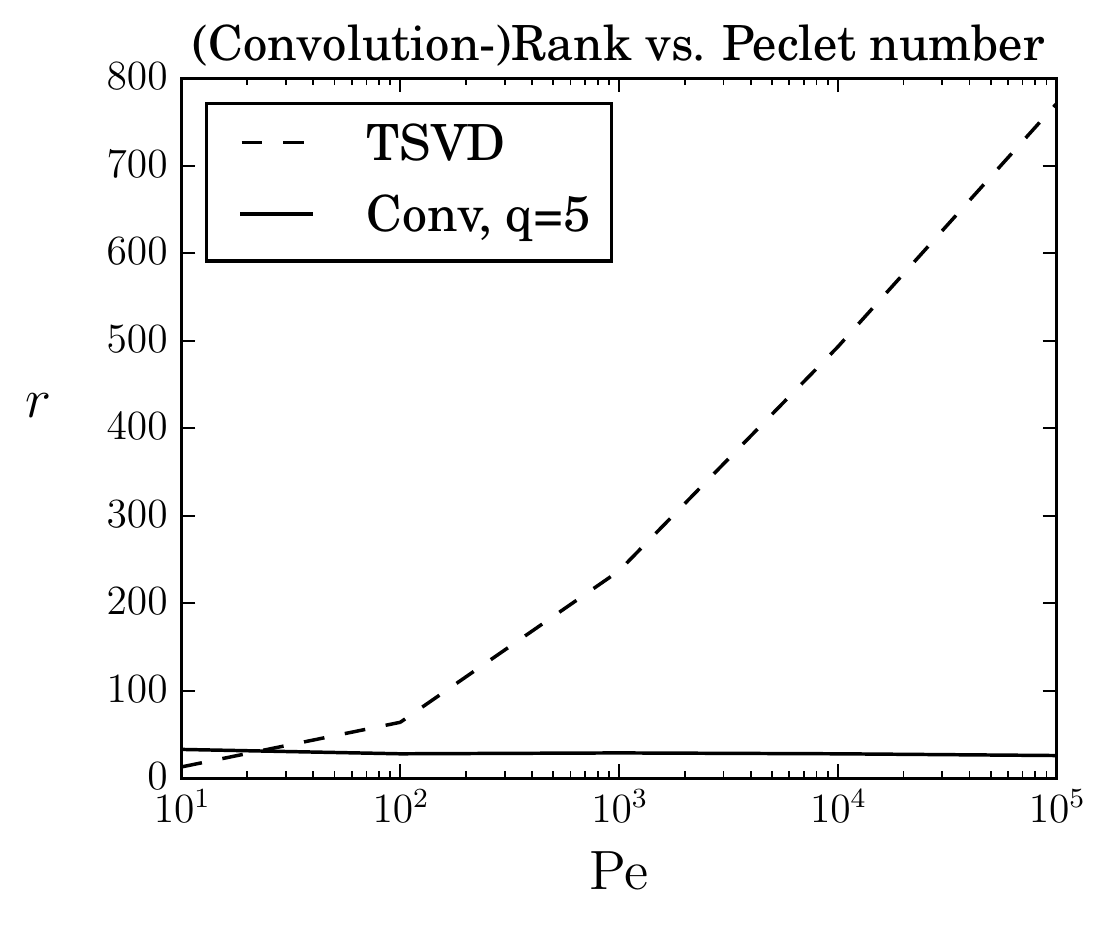}
	  \end{minipage}\hfill
	  \begin{minipage}[t]{0.43\textwidth}
		\mbox{}\\[-\baselineskip]
	    \caption{
	       \textbf{Advection-diffusion inverse problem Hessian:} The (convolution) rank, $r$, required to achieve a relative approximation error of $10$\% for a variety of Peclet numbers, $\mathrm{Pe}$. `TSVD' indicates truncated SVD low rank approximation, and 'Conv' indicates our product-convolution scheme.
	    } \label{fig:advection_diffusion_data_scalability}
	  \end{minipage}
\end{figure}

Figure \ref{fig:advection_diffusion_data_scalability} compares our scheme to TSVD for a sequence of increasing Peclet numbers, from $\mathrm{Pe}=10^1$ to $\mathrm{Pe}=10^5$. The curves show the (convolution) rank, $r$, required to achieve a relative error tolerance of $10$\% (estimated using $q=5$ random adjoint samples). Whereas the required rank for TSVD grows dramatically as $\mathrm{Pe}$ increases, the required convolution rank for our scheme remains constant.

Figure \ref{fig:advection_convergence_comparison} shows the convergence of Krylov methods for solving the Hessian linear system using GMRES with our preconditioner (`GMRES-CONV'), compared to conjugate gradient with regularization preconditioning (`CG-REG'), for $\mathrm{Pe}=10^4$. Here the product-convolution approximation is computed to a $5$\% relative error tolerance. Our preconditioner substantially outperforms regularization preconditioning, coverging rapidly even though the Peclet number is large. In Figure \ref{fig:advection_iterates_comparison}, we show intermediate reconstructions associated with $1$, $5$, and $50$ Krylov iterations, for both GMRES-CONV and CG-REG. CG-REG first reconstructs large-scale features of $m$, then medium-scale features, then small-scale features, while GMRES-CONV reconstructs features of $m$ at all scales simultaneously. Even one iteration of GMRES-CONV yields a visually reasonable reconstruction.

\begin{figure}
	\begin{minipage}[t]{0.6\textwidth}
		\mbox{}\\[-\baselineskip]
	    \includegraphics[scale=0.4]{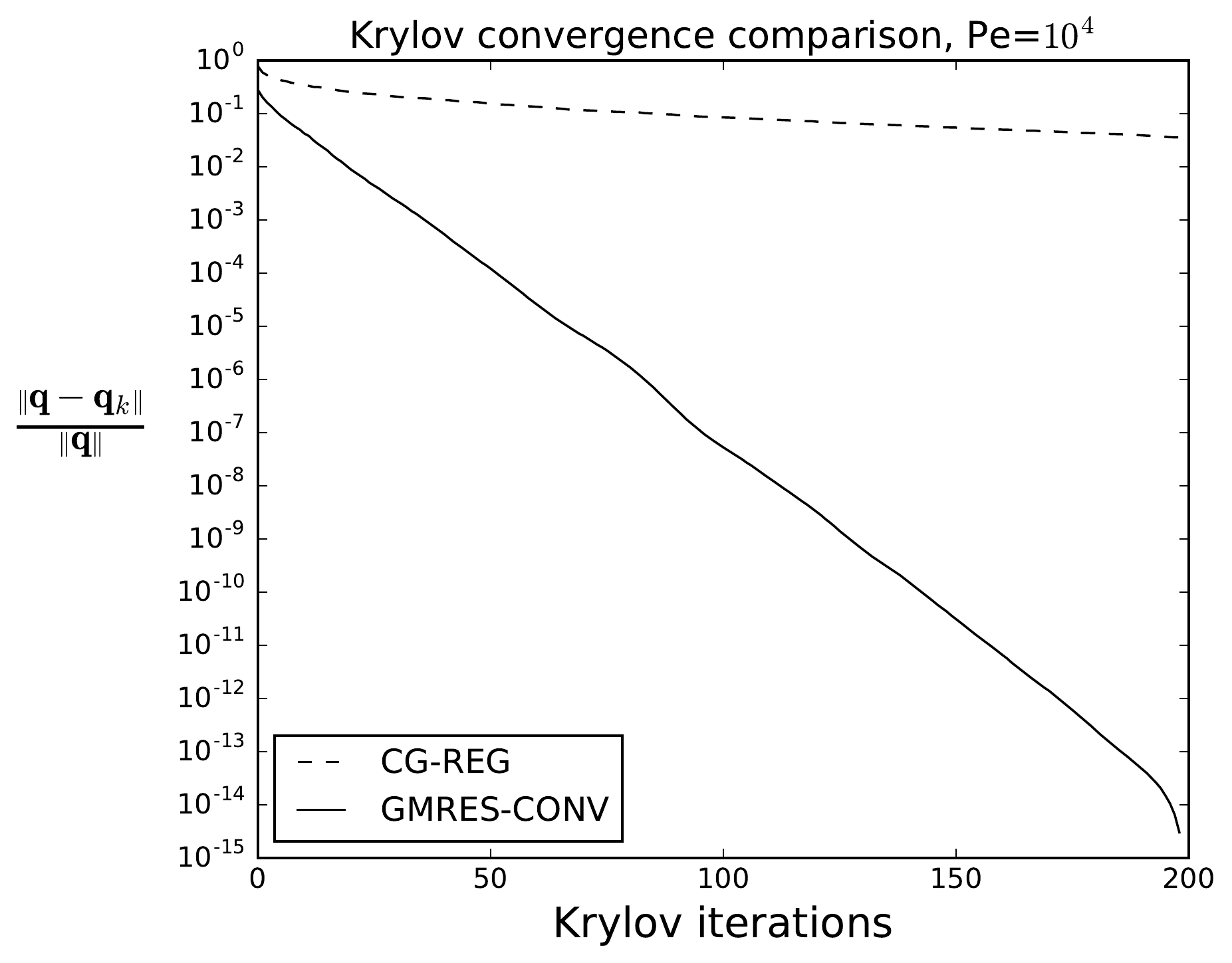}
	  \end{minipage}\hfill
	  \begin{minipage}[t]{0.39\textwidth}
		\mbox{}\\[-\baselineskip]
	    \caption{
	       \textbf{Advection-diffusion inverse problem Hessian:} Convergence of conjugate gradient with regularization preconditioning (`CG-REG'), compared to GMRES with our product-convolution preconditioner, \eqref{eq:adv_complete_preconditioner} (`GMRES-CONV'), for solving the Hessian linear system. Here $\mathrm{Pe}=10^4$, and the product-convolution approximation is accurate to $5$\% relative error.
	    } \label{fig:advection_convergence_comparison}
	  \end{minipage}
\end{figure}

\begin{figure}
	\center
    \includegraphics[scale=0.4]{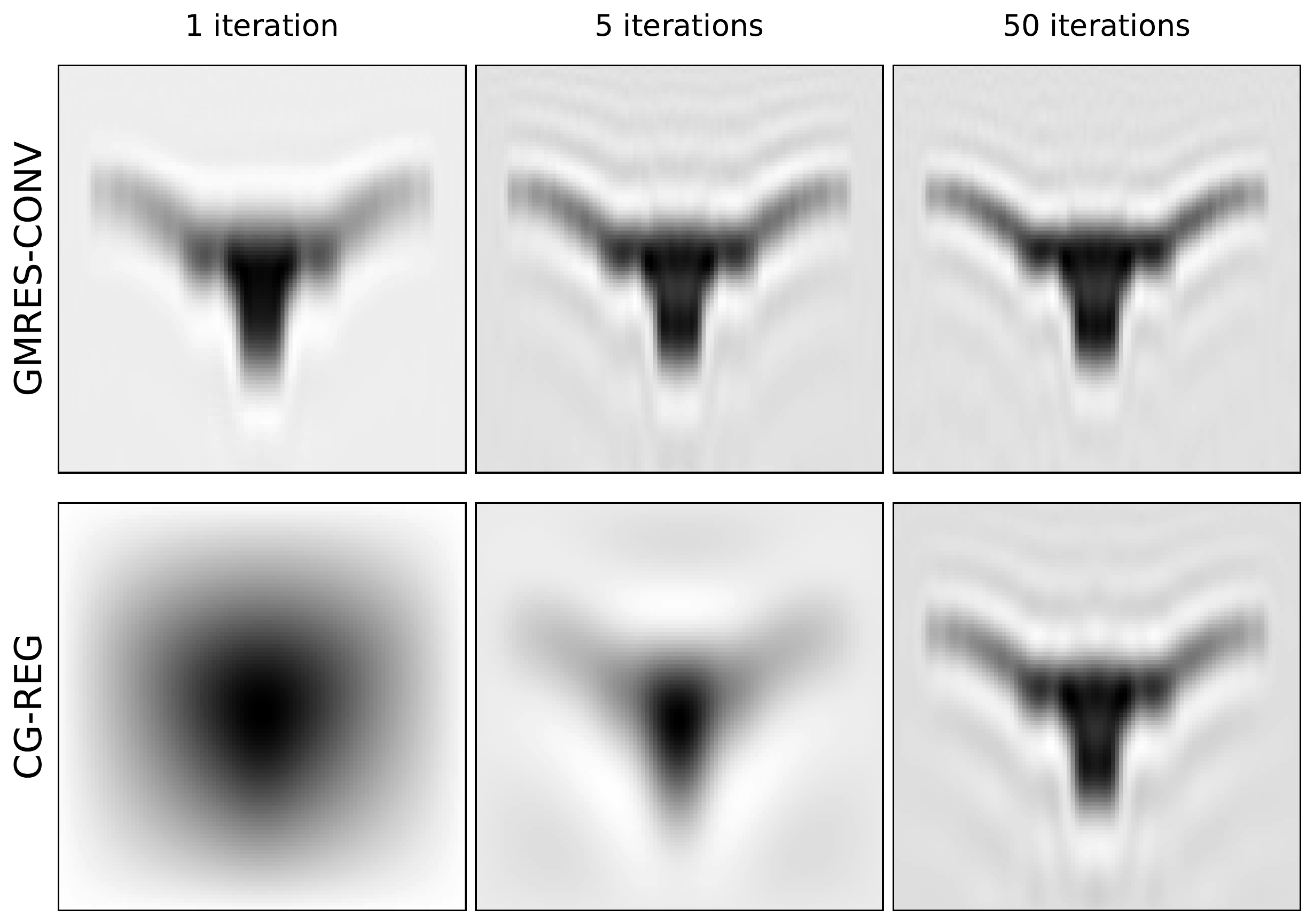}
    \caption{\textbf{Advection-diffusion inverse problem Hessian:} Comparison of parameter reconstructions associated with terminating the Krylov solver after $1$, $5$, and $50$ iterations, for both GMRES with our preconditioner (`GMRES-CONV'), and conjugate gradient with regularization preconditioning (`CG-REG'). Here $\mathrm{Pe}=10^4$, and the product-convolution approximation is accurate to $5$\% relative error.}
    \label{fig:advection_iterates_comparison}
\end{figure}

\section{Conclusions}
\label{sec:conclusion}

In this paper we presented a matrix-free adaptive grid product-convolution operator approximation scheme. The efficiency of our scheme depends on the degree to which the operator being approximated is locally translation-invariant. As a result, our scheme is well-suited for approximating or preconditioning operators that arise in Schur complement methods for solving partial differential equations (PDEs), reduced Hessians in PDE-constrained optimization and inverse problems, integral operators, covariance operators with spatially varying kernels, and Dirichlet-to-Neumann maps or other Poincaré–Steklov operators in multiphysics problems. These operators are often dense, implicitly defined, and high-rank, making them difficult to approximate with standard techniques. Our scheme is best suited to moderate accuracy requirements (say, $80$\% to $99$\% accuracy). 

Our scheme improves on existing product-convolution schemes by providing an automated method for performing adaptivity, and by addressing issues related to boundaries. Once constructed, the approximation can be manipulated efficiently and accessed in ways that the original operator cannot: matrix entries of the approximation can be computed at $O(1)$ cost, the approximation (or \emph{blocks} of the approximation) can be efficiently applied to vectors with the FFT, and the approximation can be efficiently converted to $H$-matrix format. Once in $H$-matrix format, it can be factorized, inverted, or otherwise manipulated with fast $H$-matrix arithmetic. Since our scheme is best suited to moderate accuracy requirements, the resulting $H$-matrix can be exploited to construct a good preconditioner.

We tested our scheme numerically on a spatially varying blur operator, on the non-local component of an interface Schur complement for the Poisson operator, and on the data misfit Hessian for an advection-diffusion inverse problem. We saw that our scheme outperformed existing methods in all cases. Additionally, we found that the scheme performs well even when only a handful of random samples are used to construct the a-posteriori error estimator used in the adaptive refinement procedure.

\section*{Acknowledgements}
We thank J.J. Alger, Benjamin Babcock, Joe Bishop, Andrew Potter, Georg Stadler, Umberto Villa, and Hongyu Zhu for helpful discussions. We thank the anonymous reviewers for their helpful and in-depth comments, which have helped us improve this manuscript considerably.

\FloatBarrier


\bibliographystyle{siamplain}
\bibliography{conv}
\end{document}